%% file: The_asymptoticity_of_Teichmuller_rays.tex
\documentclass[a4paper,12pt,reqno]{amsart}

\usepackage[utf8]{inputenc}
\usepackage[english]{babel}
\usepackage{amsmath,amssymb,amsfonts,amsthm}
\usepackage{mathrsfs}
\usepackage{graphicx}
\usepackage{booktabs}
\usepackage{float}
\usepackage[numbers,square,sort&compress]{natbib}
\usepackage[colorlinks,linkcolor=blue,anchorcolor=blue,citecolor=blue,pdfproducer={LaTex by Zhiyang Lyu}]{hyperref}
\usepackage[a4paper, margin = 3cm, bottom = 3cm]{geometry}

\usepackage{enumerate}
\usepackage{url}
\usepackage{pgfplots}
\usepgfplotslibrary{dateplot}

\usepackage{import}
\usepackage{pdfpages}
\usepackage{transparent}
\usepackage{xcolor}

\newtheorem{theorem}{Theorem}[section]
\newtheorem{corollary}[theorem]{Corollary}
\newtheorem{lemma}[theorem]{Lemma}
\newtheorem{proposition}[theorem]{Proposition}

\newtheorem{definition}[theorem]{Definition}

\newtheorem{example}[theorem]{Example}
\newtheorem*{xrem}{Remark}

\title[The asymptoticity of pairs of Teichm\"uller rays] {The asymptoticity of pairs of Teichm\"uller rays}

\author[G. Hu]{Guangming Hu}
\address [Guangming Hu] {College of Science, Nanjing University of Posts and Telecommunications, Nanjing, 210003, P.R. China}
\email{20230210@njupt.edu.cn}
\author[Z. Lyu]{Zhiyang Lyu}
\address[Zhiyang Lyu] {School of Mathematical Sciences, University of Science and Technology of China, 96 Jinzhai Road, 230026, Hefei, Anhui, P. R. China}
\email{lyuzhiyang@ustc.edu.cn}

\author[H. Miyachi]{Hideki Miyachi}
\address [Hideki Miyachi] {School of Mathematics and Physics, College of Science and Engineering, Kanazawa University, Kakuma-machi, Kanazawa, Ishikawa, 920-1192, Japan}
\email{miyachi@se.kanazawa-u.ac.jp}

\author[Y. Qi]{Yi Qi}
\address [Yi Qi] {School of Mathematics and Systems Science, Beihang University, Beijing, 100191, P. R. China}
\email{yiqi@buaa.edu.cn}



\thanks{This work is partially supported by NSFC Grant Numbers 12101275, 12271017, the China Scholarship Council (CSC) Grant Number 202306020157 and JSPS KAKENHI Grant Numbers 20H01800, 20K20519, 22H01125.}

\makeatletter
\hypersetup{
    pdftitle={\@title},
    pdfauthor={Guangming Hu, Zhiyang Lyu, Hideki Miyachi, Yi Qi},
    pdfsubject={\@subjclass},
    pdfkeywords={\@keywords}
}
\makeatother

\begin{document}

\baselineskip=16pt
\parskip=2pt

\begin{abstract}
	  
    In this paper, we study the limit of Teichm\"uller distance between two points along a pair of Teichm\"uller rays. We obtain an explicit formula for the limiting Teichm\"uller distance when the vertical measured foliations of the quadratic differentials are finite sums of weighted simple closed curves and uniquely ergodic measures. The limit is expressed in terms of ratios of the corresponding moduli and the Teichm\"uller distance between the limit surfaces when the vertical measured foliations are absolutely continuous. Consequently, two Teichm\"uller rays are asymptotic if and only if their vertical measured foliations are modularly equivalent and their limit surfaces coincide, which implies a main result of Masur on the asymptoticity of Teichm\"uller rays determined by uniquely ergodic quadratic differentials.  Furthermore, we prove that the infimum of the limiting Teichm\"uller distances can be represented in terms of the distance between the limit surfaces of the Teichm\"uller rays and the detour metric of their endpoints on the Gardiner-Masur boundary, when the initial points of the rays vary along the Teichm\"uller geodesics.
    
\end{abstract}

\maketitle

%\setcounter{tocdepth}{1}

%\tableofcontents

%=====================================================================

\section{Introduction}\label{Introduction}

Let $S$ be a Riemann surface of genus $g$ with $n$ punctures ($3g-3+n\geq 1$). The Teichm\"uller space $\mathcal{T}(S)$ of $S$ is the space of all marked Riemann surfaces up to Teichm\"uller equivalence. There is a natural metric $d_{\mathcal{T}}(\cdot,\cdot)$ on $\mathcal{T}(S)$, called Teichm\"uller metric. It is an important problem in history that whether the Teichm\"uller metric is of negative curvature or not. In 1975,  Masur \cite{Masur1975} showed that Teichm\"uller metric does not have negative curvature in the sense of Busemann. Moreover, Masur and  Wolf \cite{MW1995} proved in 1994 that Teichm\"uller space, equipped with Teichm\"uller metric is not Gromov hyperbolic. 

The proof of Masur in \cite{Masur1975} is based on a key result discovered by him that the Teichm\"uller distance between two Teichm\"uller geodesic rays is bounded if the corresponding vertical measured foliations are Jenkins-Strebel and topologically equivalent. This result drew out the study on the  asymptotic behavior of Teichm\"uller geodesic rays.

In 1980, Masur \cite{Mas1980} further showed that two Teichm\"uller rays are asymptotic (and therefore bounded) if the corresponding vertical measured foliations are uniquely ergodic and topologically equivalent without simple closed curve formed by  saddle connections. In 2001, Ivanov \cite{Iva2001} showed that two Teichm\"uller rays are bounded if the vertical measured foliations are absolutely continuous, and divergent if the geometric intersection of the vertical measured foliations is nonzero. In 2010, Lenzhen and Masur \cite{LM2010} proved that two Teichm\"uller rays are divergent if the vertical measured foliations are not topologically equivalent or topologically equivalent but not absolutely continuous. 

In the studying of EDM rays and the Deligne-Mundford Compactification in 2010, Farb and Masur \cite{FM2010} showed that the limit of Teichm\"uller distance  between points along two EDM (Jenkins-Strebel) rays in the moduli space exists and equals to the distance between their endpoints on the boundary of augmented moduli space. Consequently, the rays are asymptotic if their endpoints coincide. 
In 2014, Amano \cite{Ama2014, Ama2014-TheAsym} investigated the limit of Teichm\"uller distance  between points along two Jenkins-Strebel rays in Teichm\"uller space, simply called limiting Teichm\"uller distance below, and obtained an explicit formula of the limiting distance. By the formula of the limiting Teichm\"uller distance, he further showed that two Jenkins-Strebel rays are asymptotic if and only if the measured foliations are modularly equivalent and the endpoints of these rays are the same in the augmented Teichm\"uller space.

Based on the work of Amano \cite{Ama2014, Ama2014-TheAsym}, Lenzhen and Mausur \cite{LM2010} and Ivanov \cite{Iva2001}, One may naturally propose the following problem.

\medskip
\noindent{\bf  Problem:} {\em For any two absolutely continuous  Teichm\"uller geodesic rays in the Teichm\"uller space, does the limit of Teichm\"uller distance between points along these rays exist? Is there also an explicit formula for the limiting Teichm\"uller distance? }
\medskip

The main goal of this paper is to study this problem. We give an affirmative answer for pairs of Teichm\"uller gedesic rays whose corresponding vertical measured foliations can be expressed as finite sums of weighted simple closed curves and uniquely ergodic measures. 

%%%%%%%%%%%%%%%%%%%%%%
To introduce our main results, we need some simple preparations and notions. It is known that the  end point or limit surface of a Jenkins-Strebel ray is a noded Riemann surface in the augmented Teichm\"uller space. The limit surface of a general Teichm\"uller ray was studied and referred to as the conformal limit of the Teichm\"uller ray by Gupta \cite{Gup2019} recently. For a Teichm\"uller ray $\mathcal{R}_{q,X}(t)$ induced by a unit norm holomorphic quadratic differential $q$ on $X\in\mathcal{T}(S)$, the conformal limit is a disjoint union of punctured Riemann surfaces. Each of these surfaces is associated with a connected component of the finite critical graph $\Gamma_q$ and is formed by attaching half planes and semi-infinite cylinders to the critical graph of $q$. These surfaces have infinite area under the singular flat metrics determined by $q$, which are referred to as half-plane structures (see \cite{Gup2014} and \cite{GW2016}). The limit surface of a Teichm\"uller ray can be understood as the pointed Gromov-Hausdorff limit by choosing a set of singularities from each connected component of $\Gamma_q$ as  basepoints. 
In this paper, we provide a detail construction of the limit surface and prove that the Teichm\"uller ray convergents to it in the pointed Gromov-Hausdorff sense. Furthermore we define the Teichm\"uller distance $d_{\overline{\mathcal{T}}}$ between the limit surfaces (See \S\ref{limit surfaces} for details).  
%%%%%%%%%%%%%%%%%%%%%%%%%%%

Now we can state our main result as follows precisely.

\begin{theorem}\label{TheLimitDistance}
Let $\mathcal{R}_{q,X}(t)$ and $\mathcal{R}_{q^\prime,Y}(t)$ be two Teichm\"uller rays, and let $V(q)$ and $V(q^\prime)$ $(H(q)$ and $H(q^\prime))$ denote the vertical $($horizontal$)$ measured foliations induced by quadratic differentials $q$ and $q^\prime$, respectively. Suppose $V(q)$ can be expressed as $V(q)=\sum_{j=1}^{N}a_jG_j$, where $a_j$ is a positive number and $G_j$ is either a simple closed curve or a uniquely ergodic measure.
\begin{itemize}
\item[(i)] If  $V(q)$ and $V(q^\prime)$ are absolutely continuous $($i.e. $V(q^\prime)=\sum_{j=1}^{N}b_{j}G_{j}$ with $b_j>0$$)$, then the limiting Teichm\"uller distance exists and
$$
\lim_{t\to\infty}d_{\mathcal{T}}(X_t,Y_t)=\max\left\{\frac{1}{2}\log\max_{1\leq j\leq N}\left\{\frac{m^\prime_j}{m_j},\frac{m_j}{m^\prime_j}\right\},d_{\overline{\mathcal{T}}}(X_\infty,Y_\infty)\right\},
$$
where $m_j=\frac{a_j}{i(G_j,H(q))}$ and $m^\prime_j=\frac{b_j}{i(G_j,H(q^\prime))}$ are the modulus of $G_j$ on $X$ and $Y$, $X_\infty$ and $Y_\infty$ are the limit surfaces of $\mathcal{R}_{q,X}(t)$ and $\mathcal{R}_{q^\prime,Y}(t)$, respectively.
\item[(ii)] Otherwise, 
$$
\lim_{t\to\infty}d_{\mathcal{T}}(X_t,Y_t)=+\infty.
$$
\end{itemize}
\end{theorem}

Theorem \ref{TheLimitDistance} is a generalization of the main result in \cite{Ama2014-TheAsym}, since the Teichm\"uller distance between the limit surfaces defined here is equal to that in \cite{Ama2014-TheAsym} for Jenkins-Strebel rays (see \S\ref{Teichmuller rays} and \S\ref{limit surfaces} for details).

Furthermore, we obtain a necessary and sufficient condition for the asymptoticity of two Teichm\"uller rays.

\begin{corollary}\label{TheAsymptoticCondition}
    Under the assumption of Theorem \ref{TheLimitDistance}, the Teichm\"uller rays $\mathcal{R}_{q,X}(t)$ and $\mathcal{R}_{q^\prime,Y}(t)$ are asymptotic if and only if the vertical measured foliations $V(q)$ and $V(q^\prime)$ are modularly equivalent and $X_\infty=Y_\infty$. 
\end{corollary}

Moreover, we recover the main result of Masur in \cite{Mas1980} by Corollary \ref{TheAsymptoticCondition}. Masur showed the asymptotic behavior of Teichm\"uller rays determined by uniquely ergodic measured foliations, under the notable condition that there are no simply closed curves consisting of saddle connections. From our construction of the limit surface for a Teichm\"uller ray, this condition implies that the limit surface is a disjoint union of punctured spheres. Thus, their asymptoticity follows directly from Corollary \ref{TheAsymptoticCondition}.

\begin{corollary}[Masur \cite{Mas1980}]\label{Masur'sTheorem}
    Let $\mathcal{R}_{q,X}(t)$ be a Teichm\"uller ray in Teichm\"uller space of genus $g>1$. The vertical measured foliation $V(q)$ is uniformly ergodic on $X$ and the finite critical graph $\Gamma_q$ contains no simple closed curves. Then for any $Y$ not on $\mathcal{R}_{q,X}(t)$, there is a Teichm\"uller ray through $Y$ asymptotic to $\mathcal{R}_{q,X}(t)$.
\end{corollary}

Since the limiting distance depends on ratios of the moduli determined by the holomorphic quadratic differentials on the initial points, we can consider the infimum of the limiting distances when the initial points shift along the Teichm\"uller rays. It is shown that the infimum is represented by the detour metric $\delta$ between the end points of the Theich\"uller rays on the Gardiner-Masur boundary of $\mathcal{T}(S)$ and the distance between their limit surfaces.

\begin{proposition}\label{InfimumofLimitDistance}
    Under the assumption of Theorem \ref{TheLimitDistance}, if the vertical measured foliations $V(q)$ and $V(q^\prime)$ are absolutely continuous, then by shifting the starting points of $\mathcal{R}_{q,X}(t)$ and $\mathcal{R}_{q^\prime,Y}(t)$, the minimum of the limiting distances is 
    $$
    \max\left\{\frac{1}{2}\delta(\hat{\mathcal{E}}_{q,X},\hat{\mathcal{E}}_{q^\prime,Y}),d_{\overline{\mathcal{T}}}(X_\infty,Y_\infty)\right\},
    $$
    where $\delta$ is the detour metric and $\hat{\mathcal{E}}_{q,X}$, $\hat{\mathcal{E}}_{q^\prime,Y}$ are the end points of $\mathcal{R}_{q,X}(t)$ and $\mathcal{R}_{q^\prime,Y}(t)$ on the Gardiner-Masur boundary of $\mathcal{T}(S)$, respectively.
\end{proposition}

This paper is organized as follows. In section $2$, we recall some relevant background, notions and basic results on Teichm\"uller spaces, quadratic differentials, measured foliations and Teichm\"uller rays. In section $3$, we provide a concrete construction of the limit surface of a Teichm\"uller ray in details and prove that the Teichm\"uller ray convergents to it in the pointed Gromov-Hausdorff sense. Furthermore we define the Teichm\"uller distance $d_{\overline{\mathcal{T}}}$ between limit surfaces. In section $4$, we give the upper estimate of the limiting Teichm\"uller distance by constructing quasiconformal mappings. In section $5$, we give the lower estimate of the limiting Teichm\"uller distance and complete proofs of Theorem \ref{TheLimitDistance}, Corollary \ref{TheAsymptoticCondition} and \ref{Masur'sTheorem}. In section $6$, we prove Proposition \ref{InfimumofLimitDistance}.

\section{Preliminaries}

\subsection{Teichm\"uller spaces}

Let $S$ be a Riemann surface of genus $g$ with $n$ punctures such that $3g-3+n\geq 1$. A \emph{marked Riemann surface} denoted by $(X,f)$ is a pair of a Riemann surface $X$ and a quasiconformal mapping $f:S\to X$ called the \emph{marking} of $X$. Two marked Riemann surfaces $(X_1,f_1)$ and $(X_2,f_2)$ are \emph{Teichm\"uller equivalent} if there is a conformal mapping $h:X_1\to X_2$ such that $f_2$ is homotopic to $h\circ f_1$. The \emph{Teichm\"uller space} $\mathcal{T}(S)$ of $S$ is the space of all Teichm\"uller equivalent classes $[X,f]$ containing $(X,f)$. We will use the Riemann surface $X$ to denote the $[X,f]\in\mathcal{T}(S)$ for simplicity. There is a complete metric called \emph{Teichm\"uller metric} $d_{\mathcal{T}}$ on $\mathcal{T}(S)$. For any two $X_1, X_2\in\mathcal{T}(S)$, the Teichm\"uller distance is defined by 
$$
d_{\mathcal{T}}(X_1,X_2)=\frac{1}{2}\inf_{h}\{\log K(h)\},
$$
where the infimum is over all quasiconformal mapping $h:X_1\to X_2$ such that $f_2$ is homotopic to $h\circ f_1$, and $K(h)$ is the maximal quasiconformal dilatation of $h$.

A \emph{noded Riemann surface} $R$ is a connected Hausdorff space with a set $P$ of finitely many distinguished points such that each connected component of $R\setminus P$ is a Riemann surface of finite type, and each point $p_k\in P$ called a \emph{node} of $R$ has a neighborhood which is biholomorphic to 
$$
\{(z,w)\in\mathbb{C}^2\mid zw=0,|z|<1,|w|<1\},
$$
where $p_k$ is mapped to $(0,0)\in\mathbb{C}^2$. It is clear that a Riemann surface $X$ is a noded Riemann surface without nodes.

The \emph{augmented Teichm\"uller space} $\widehat{\mathcal{T}}(S)$ is the space of all Teichm\"uller equivalent classes $[R,f]$ of marked noded Riemann surface $(R,f)$, where $R$ is a noded Riemann surface, and $f:S\to R$ is a continuous mapping such that some disjoint simple closed curves on $S$ are contracted to the nodes of $R$, and $f$ is a homeomorphism on the complement of the simple closed curves. Two noded Riemann surfaces $(R_1,f_1)$ and $(R_2,f_2)$ are Teichm\"uller equivalent if there is a homeomorphism $h:R_1\to R_2$ such that $h\circ f_1$ is homotopic to $f_2$, where the restriction of $h$ to a component of $R_1\setminus\{\text{nodes of }R_1\}$ onto a component of $R_2\setminus\{\text{nodes of }R_2\}$ is conformal (see \cite{Abi1977}).

\subsection{Quadratic differentials}

A \emph{quadratic differential} $q$ on a Riemann surface $X$ is a tensor of the form $q(z)dz^2$ where $q(z)$ is a function of a local coordinate on $X$. We call $q$ a \emph{holomorphic quadratic differential} when $q(z)$ is a holomorphic function with at most simple poles at the punctures of $X$. The zeros and poles of $q$ are called the \emph{critical points} of $q$, and others are called the \emph{regular points} of $q$. For a holomorphic quadratic differential $q$, there are finitely many critical points of $q$ on $X$, and the norm $\|q\|=\iint_X |q|dxdy$ is finite. A holomorphic quadratic differential $q$ is called that of unit norm if $\|q\|=1$.

If a maximal smooth arc $z=\gamma(t)$ on $X$ satisfies $q(\gamma(t))\gamma^\prime(t)^2>0$, the arc is a \emph{horizontal trajectory} of $q$, and the arc is a \emph{vertical trajectory} of $q$ if it satisfies $q(\gamma(t))\gamma^\prime(t)^2<0$. A \emph{critical trajectory} of $q$ is either a vertical trajectory connecting two critical points of $q$ or a vertical trajectory with an endpoint at a critical point of $q$. Let $\widetilde{\Gamma}_q$ be the union of critical points, punctures, critical trajectories and vertical trajectories with endpoints at the punctures on $X$, which is called the \emph{critical graph} of $q$. The set of critical points, punctures and vertical trajectories connecting critical or punctures on $X$ is denoted by $\Gamma_q$ and is called the \emph{finite critical graph} of $q$. The finite critical graph $\Gamma_q$ is a subset of the critical graph $\widetilde{\Gamma}_q$.

For a holomorphic quadratic differential $q$, It is known that the components of $X\setminus\Gamma_q$ consist of finitely many cylinders and minimal domains, where each cylinder is swept out by simple closed vertical trajectories of $q$, and a minimal domain is a domain on $X$ in which all vertical trajectories are dense. A quadratic differential is called a \emph{Jenkins-Strebel differential} if the components of $X\setminus\Gamma_q$ are all cylinders.

\subsection{Measured foliations}

A \emph{measured foliation} $(\mathcal{F},\mu)$ on surface $S$ is a singular foliation $\mathcal{F}$ with transverse measure $\mu$. Let $\mathcal{S}$ be the set of homotopic classes of non-trivial and non-peripheral simple closed curves on $S$. We can define the intersection number of a measured foliation $(\mathcal{F},\mu)$ and a $\alpha\in\mathcal{S}$ as
$$
i((\mathcal{F},\mu),\alpha)=\inf_{\alpha^\prime\in\alpha}\int_{\alpha^\prime}d\mu,
$$
where the infimum is taken over all simple closed curves $\alpha^\prime$ in $\alpha$. Two measured foliations $(\mathcal{F}_1,\mu_1)$ and $(\mathcal{F}_2,\mu_2)$ are equivalent if 
$$
i((\mathcal{F}_1,\mu_1),\alpha)=i((\mathcal{F}_2,\mu_2),\alpha)
$$
holds for all $\alpha\in\mathcal{S}$. Let $\mathcal{F}=[\mathcal{F},\mu]$ be the equivalent class containing $(\mathcal{F},\mu)$, and We denoted by $\mathcal{MF}(S)$ the space of equivalent classes of measure foliations on $S$. The space $\mathcal{MF}(S)$ has the weak topology induced by the intersection number functions in $\mathbb{R}^{\mathcal{S}}_{\geq 0}$. The set of weighted simple closed curves $\mathbb{R}_{\geq 0}\otimes\mathcal{S}$ is dense in $\mathcal{MF}(S)$. Then the intersection number can extend continuously to an intersection function on $\mathcal{MF}(S)\times\mathcal{MF}(S)$ (cf. \cite{Bon1986}, \cite{Bon1988} and \cite{Ree1981}).

For a holomorphic quadratic differential $q$ on Riemann surface $X$, each regular point of $q$ has a canonical coordinate $z=x+iy$ such that $q=dz^2$ in the coordinate, and the vertical trajectory through the regular point is a vertical line in the canonical coordinate. There is a \emph{vertical measured foliation} $V(q)$ determined by $q$ on $X$, where the singular foliation of $V(q)$ is formed by the vertical trajectories of $q$ and the transverse measure is induced by $|dx|$. The \emph{singularities} of $V(q)$ are the critical points of $q$ and the punctures on $X$. The vertical trajectories of $q$ are called the \emph{leaves} of $V(q)$, and the vertical trajectories joining two critical or punctures are called the \emph{saddle connections} of $V(q)$. Similarly, there is also a \emph{horizontal measured foliation} $H(q)$ on $X$ induced by $q$. Hubbard and Masur \cite{HM1979} showed that for each measured foliation $[\mathcal{F},\mu]\in\mathcal{MF}(X)$, there exists a holomorphic quadratic differential $q$ on $X$ such that $V(q)\in [\mathcal{F},\mu]$.

The vertical measured foliation $V(q)$ on a minimal component $\Omega$ of $X\setminus\Gamma_q$ can be represented as 
$$
V(q)\big|_\Omega=\sum_{i=1}^{p}b_i\mu_i,
$$
where $b_i\geq 0$ and $\{\mu_i\}$ is a set of projectively-distinct ergodic transverse measures. The $p$ is bounded, which depends only on the topology of the surface $X$. The transverse measure of $V(q)$ on a minimal component $\Omega$ is said to be \emph{uniquely ergodic} if it is unique up to scalar multiplication. The restriction of $V(q)$ to a cylinder $A$ in $X\setminus\Gamma_q$ can be represented as $V(q)\big|_A=b\alpha$, where $b>0$ is the height of the cylinder $A$, and $\alpha$ is a simple closed curve on $A$ which is homotopic to the closed leaf of $V(q)$ sweeping out the cylinder $A$. Thus, the vertical measured foliation $V(q)$ on $X$ can be written as 
$$
V(q)=\sum_{j=1}^{N}b_jG_j,
$$
where $G_j$ is a simple closed curve or an ergodic measure on $X$. When $G_j$ is an ergodic measure $\mu_j$, for simplicity, we also consider $G_j$ as the corresponding singular foliation $G_j$ with the ergodic measure $\mu_j$ on $X$.

Let $V(q)$ be a vertical measured foliation on a Riemann surface $X=[X,f_1]\in\mathcal{T}(S)$ and $V(q^\prime)$ be a vertical measured foliation on a Riemann surface $Y=[Y,f_2]\in\mathcal{T}(S)$. The measured foliations $V(q)$ and $V(q^\prime)$ are \emph{topologically equivalent} if there is a homeomorphism $h:X\setminus\Gamma_q\to Y\setminus\Gamma_{q^\prime}$ such that $h$ is homotopic to the mapping $f_2\circ f_1^{-1}$ restricting to $X\setminus\Gamma_q$, and $h$ takes the leaves of $V(q)$ to the leaves of $V(q^\prime)$. We say that the measured foliations $V(q)$ and $V(q^\prime)$ are \emph{absolutely continuous} if they are topologically equivalent and if we can write the measured foliations $V(q^\prime)$ and $h_\ast(V(q))$ as 
$$
V(q^\prime)=\sum_{j=1}^{N}b_jG_j, \quad h_\ast(V(q))=\sum_{j=1}^{N}a_jG_j,
$$
where $G_j$ is a simple closed curve or an ergodic measure on $Y$ and $a_j$ and $b_j$ are positive real numbers. For simplicity, we also write $V(q)$ as $V(q)=\sum_{j=1}^{N}a_jG_j$ and consider each $G_j$ as the corresponding simple closed curve or ergodic measure on $X$.

For a vertical measured foliation $V(q)=\sum_{j=1}^{N}a_jG_j$ on $X$, let 
$$
m_j=\frac{a_j}{i(G_j,H(q))},
$$
which is called the \emph{modulus} of $G_j$ on $X$. We say that $V(q)=\sum_{j=1}^{N}a_jG_j$ and $V(q^\prime)=\sum_{j=1}^{N}b_jG_j$ are \emph{modularly equivalent} if for all $j$,
$$
\frac{a_j}{i(G_j,H(q))}=C\frac{b_j}{i(G_j,H(q^\prime))},
$$
where $C$ is a positive constant independent of $j$.

\subsection{Teichm\"uller rays}\label{Teichmuller rays}
A quasiconformal mapping $f$ on $X$ is called a \emph{Teichm\"uller mapping} if the Beltrami coefficient $\mu_f$ is of the form $\mu_f=\frac{K(f)-1}{K(f)+1}\frac{\bar{q}}{|q|}$, where the $q$ is a unit norm holomorphic quadratic differential on $X$.

The \emph{extremal quasiconformal mapping} $g$ between two Riemann surface is a mapping whose dilatation $K(g)$ attains the infimum of the dilatation of quasiconformal mapping homotopic to $g$. Teichm\"uller's theorem states that, for any two surfaces $X, Y\in\mathcal{T}(S)$, there exists a unique extremal quasiconformal mapping between $X$ and $Y$, which is the Teichm\"uller mapping $f$ for a unique unit norm holomorphic quadratic differential $q$ on $X$. Then the dilatation $K(f)$ of Teichm\"uller mapping $f$ realizes the Teichm\"uller distance $d_{\mathcal{T}}(X,Y)$.

Let $q$ be a unit norm holomorphic quadratic differential on $X$ and $f_{q,t}:X\to X_t$ be the Teichm\"uller mapping for $q$. There is a unit norm holomorphic quadratic differential $q_t$ on $X_t$ such that in the canonical coordinate $z=x+iy$ of $q$ and the canonical coordinate of $q_t$, the mapping $f_{q,t}$ is given by 
$$
z\mapsto e^tx+ie^{-t}y,
$$
where $e^t=K(f_{q,t})^{\frac{1}{2}}$. We consider the holomorphic quadratic differential $e^{2t}q_t$ on $X_t$. Thus, the Teichm\"uller mapping $f_{q,t}:X\to X_t$ is given by $z\mapsto e^{2t}x+iy$ in the canonical coordinates of $q$ and $e^{2t}q_t$. Then under the mapping $f_{q,t}$, the leaves of $H(q)$ are stretched by a factor of $e^{2t}$, while the leaves of $V(q)$ remain unchanged.  

The \emph{Teichm\"uller ray} $\mathcal{R}_{q,X}(t)$ induced by a unit norm holomorphic quadratic differential $q$ with initial point $X$ is defined by
\begin{equation*}
    \begin{array}{cccl}
        \mathcal{R}_{q,X}: & \mathbb{R}_{\geq 0} & \to & \mathcal{T}(S) \\
         & t & \mapsto & X_t=f_{q,t}(X), \\
    \end{array}
\end{equation*} 
where $f_{q,t}:X\to X_t$ is the Teichm\"uller mapping for the holomorphic quadratic differential $q$ on $X$.

A Teichm\"uller ray $\mathcal{R}_{q,R}(t)$ is called a \emph{Jenkins-strebel ray} if $q$ is a Jenkins-Strebel differential. A Jenkins-Strebel ray $\mathcal{R}_{q,R}(t)$ on $\mathcal{T}(S)$ converges to a noded Riemann surface $R_\infty$ in $\widehat{\mathcal{T}}(S)$ as $t\to\infty$ (cf. \cite{HS2007}). Let $\mathcal{R}_{q,R}(t)$ and $\mathcal{R}_{q^\prime,R^\prime}(t)$ be two Jenkins-Strebel rays with initial points $R=[R,f]$ and $R^\prime=[R^\prime,f^\prime]$, converging to $R_\infty$ and $R^\prime_\infty$ respectively. Suppose that the measured foliations $V(q)$ and $V(q^\prime)$ are absolutely continuous. There exists a homeomorphism $h:S\setminus f^{-1}(\Gamma_q)\to S\setminus {f^\prime}^{-1}(\Gamma_{q^\prime})$, homotopic to the identity, such that the mapping $f^\prime\circ h\circ f$ maps the leaves of $V(q)$ to the leaves of $V(q^\prime)$. Let $f_\infty:R\to R_\infty$ and $f^\prime_\infty:R^\prime\to R^\prime_\infty$ be two continuous mappings that contract the core curves of the cylinders in $R\setminus\Gamma_q$ and $R^\prime\setminus\Gamma_{q^\prime}$ to the corresponding nodes of $R_\infty$ and $R^\prime_\infty$, respectively. There exists a decomposition of $R_\infty\setminus\{\text{nodes of }R_\infty\}$ given by 
$$
R_\infty\setminus\{\text{nodes of }R_\infty\}=\bigcup_{i=1}^n R_{\infty,i},
$$
where each $R_{\infty,i}$ is a connected component. The surface $R^\prime_\infty\setminus\{\text{nodes of }R^\prime_\infty\}$ admits a corresponding decomposition 
$$
R^\prime_\infty\setminus\{\text{nodes of }R^\prime_\infty\}=\bigcup_{i=1}^n R^\prime_{\infty,i}
$$
satisfying
$$
(f^\prime_\infty\circ f^\prime)\circ h\circ (f_\infty\circ f)^{-1}(R_{\infty,i})=R^\prime_{\infty,i}
$$
for all $i=1,\cdots,n$. The Teichm\"uller distance between $R_\infty$ and $R^\prime_\infty$ is defined as 
$$
d_{\widehat{\mathcal{T}}}(R_\infty,R^\prime_\infty)=\max_{1\leq i\leq n}\frac{1}{2}\log\inf K(h_i),
$$
where the infimum is taken over all quasiconformal mappings $h_i:R_{\infty,i}\to R^\prime_{\infty,i}$ homotopic to the restriction of $(f^\prime_\infty\circ f^\prime)\circ h\circ (f_\infty\circ f)^{-1}$ to $R_{\infty,i}$.

Let $\mathcal{R}_{q,X}(t)$ and $\mathcal{R}_{q^\prime,Y}(t)$ be two Teichm\"uller rays. We call $\mathcal{R}_{q,X}(t)$ and $\mathcal{R}_{q^\prime,Y}(t)$ \emph{divergent} if $d_{\mathcal{T}}(X_t,Y_t)\to +\infty$ as $t\to\infty$. The rays $\mathcal{R}_{q,X}(t)$ and $\mathcal{R}_{q^\prime,Y}(t)$ are \emph{bounded} if there is a constant $M>0$ such that $d_{\mathcal{T}}(X_t,Y_t)<M$ for any $t\geq 0$. If there is 
$$
\lim_{t\to\infty}\inf_{Y^\prime\in\mathcal{R}_{q^\prime,Y}(t)}d_{\mathcal{T}}(X_t,Y^\prime)=0,
$$
$\mathcal{R}_{q,X}(t)$ and $\mathcal{R}_{q^\prime,Y}(t)$ are \emph{asymptotic}. In the asymptotic case, there is a $\sigma\in\mathbb{R}$ such that $d_{\mathcal{T}}(X_t,Y_{t+\sigma})\to 0$ as $t\to\infty$.

\section{The limit surfaces for Teichm\"uller rays}

Let $q$ be a unit norm holomorphic quadratic differential on Riemann surface $X\in\mathcal{T}(S)$ and $\mathcal{R}_{q,X}(t)$ be the Teichm\"uller ray with initial point $X$ induced by $q$. A holomorphic quadratic differential $q$ on a Riemann surface $X$ defines a singular flat metric on the surface, where the singularities are the critical points of $q$ and the punctures of $X$. We equip the surface $X_t\in\mathcal{R}_{q,X}(t)$ with the normalized singular flat metric induced by $e^{2t}q_t$ and discuss the convergence behavior of $X_t$ with the normalized singular flat metric in the Gromov-Hausdorff sense as $t\to\infty$.

\subsection{The rectangular decomposition}\label{TheDecomposition}

Let $\Omega$ be a minimal component of $X\setminus\Gamma_q$. We consider the restriction of the vertical measured foliation $V(q)$
to the region $\Omega$ and select a small horizontal segment $\tau$ along a leaf of $H(q)$ within $\Omega$. This segment $\tau$ is chosen to avoid singularities and to have no intersections with any saddle connection of $V(q)$. We label the two sides of $\tau$ as $\tau_+$ and $\tau_-$. Since each leaf of $V(q)$ is dense in $\Omega$ and there are only finitely many singularities, a leaf leaving a point on $\tau$ from the side $\tau_+$ will either reach a singularity of $V(q)$ or return to $\tau$ on either the $\tau_+$ or $\tau_-$ side. The same holds for a leaf departing from a point on $\tau$ from the $\tau_-$ side. Considering the first return of leaves leaving from $\tau$, we can define a mapping $T:\tau_+\cup\tau_-\to\tau_+\cup\tau_-$. For any $x\in\tau_+\cup\tau_-$, $T(x)$ is the first point where the leaf, starting from $x$, returns to $\tau$. Thus, $\Omega$ decomposes into finitely many rectangles, as shown in Figure \ref{rectangles}. Since $\tau$ contains no singularities and has no intersection with any saddle connections of $V(q)$, all the singularities and saddle connections lie along the vertical edges of the rectangles.

If both vertical edges of a rectangle contain singularities or saddle connections, we split the rectangle into two smaller rectangles of equal width along a leaf of $V(q)$. Then there is a decomposition of $\Omega$ such that $\Omega$ is a union of rectangles $R_i$:
$$
\Omega=R_1\cup R_2\cup\cdots\cup R_m,
$$
where each $R_i$ has only one vertical edge containing singularities or saddle connections.

\begin{figure}[!htpb]
    \centering
    \fontsize{7pt}{9pt}\selectfont
    \def\svgwidth{0.85\columnwidth}
    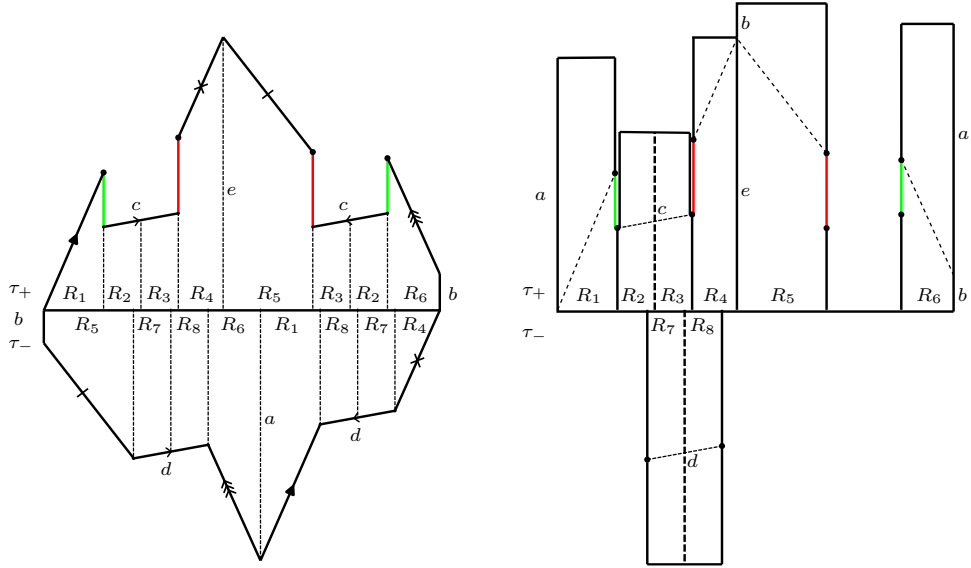

    \caption{The rectangular decomposition of a Riemann surface of genus $2$ by the first return map on a horizontal segment $\tau$. The rectangles $R_2$ and $R_3$ are formed by splitting a rectangle containing singularities on both vertical edges. The rectangles $R_7$ and $R_8$ are the same case.}
    \label{rectangles}
\end{figure}

Let $A$ be a cylindrical component of $X\setminus\Gamma_q$. The two boundaries of $A$ consist of a finite number of saddle connections of $V(q)$. We divide the cylinder $A$ into two smaller cylinders along a closed leaf of $V(q)$ such that both cylinders have the same height, and each cylinder has only one boundary consisting of saddle connections.

Each of the cylindrical and minimal components of $V(q)$ has a decomposition as described above. Then the Riemann surface $X$ decomposes into finitely many cylinders and rectangles. We glue the cylinders and rectangles along their edges containing singularities or saddle connections, as shown in Figure \ref{decomposition}. Then by gluing the cylinders and rectangles, we obtain a finite number of surfaces with boundaries, which form a decomposition of $X$:
$$
X=X_1\cup X_2\cup\cdots\cup X_n,
$$
where the number $n$ of the subsurfaces depends on the Riemann surface $X$ and $q$. 

Since the cylinders and rectangles are glued along their edges containing singularities or saddle connections, each subsurface $X_i$ in the decomposition of $X$ corresponds to a connected subgraph of the finite critical graph $\Gamma_q$. Therefore, the number $n$ of the subsurfaces is equal to the number of connected subgraph of $\Gamma_q$. Then the finite critical graph $\Gamma_q$ has a decomposition given by:
$$
\Gamma_q=\Gamma_{q,1}\cup\Gamma_{q,2}\cup\cdots\cup\Gamma_{q,n},
$$
where each $\Gamma_{q,i}$ is a connected subgraph of $\Gamma_q$ such that $\Gamma_{q,i}$ is contained in the subsurface $X_i$.

\begin{figure}[!htpb]
    \centering
    \fontsize{7pt}{9pt}\selectfont
    \def\svgwidth{0.95\columnwidth}
    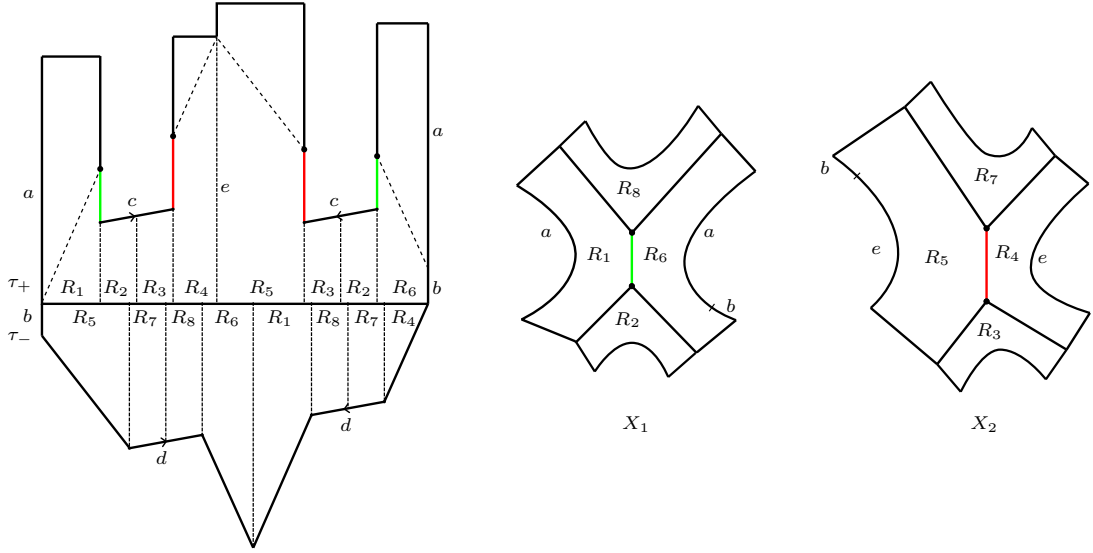

    \caption{Glue the rectangles along their edges containing singularities. The Riemann surface is divided into two subsurfaces which depend on the connected subgraph of the finite critical graph.}
    \label{decomposition}
\end{figure}

Since the graphs $\Gamma_q$ and $\widetilde{\Gamma}_q$ on $X$ are preserved along the Teichm\"uller ray $\mathcal{R}_{q,X}(t)$, we continue to denote by $\Gamma_q$ and $\widetilde{\Gamma}_q$ the corresponding graphs on the surface $X_t$ along the Teichm\"uller ray $\mathcal{R}_{q,X}(t)$.
For a surface $X_t$ on the Teichm\"uller ray $\mathcal{R}_{q,X}(t)$, the selected horizontal segment $\tau$ in a minimal component of $X\setminus\Gamma_q$ corresponds to a horizontal segment in the corresponding minimal component of $X_t\setminus\Gamma_q$, while the length of the segment on $X_t$ is multiplied by $e^{2t}$ under the normalized singular flat metric induced by $e^{2t}q_t$. For simplicity, we still denote by $\tau$ the corresponding horizontal segment on $X_t$.
Then for the decomposition of $X$, there is an analogous decomposition of $X_t$, with the width of each rectangle and the height of each cylinder are multiplied by $e^{2t}$ under the normalized singular flat metric. The decomposition of $X_t$ is written as follows:
$$
X_t=X_{t,1}\cup X_{t,2}\cup\cdots\cup X_{t,n}.
$$

The decomposition of $X$ depends on the choice of the horizontal segment $\tau$ on each minimal component $\Omega$. We choose a horizontal segment $\tau$ on each minimal component of $X\setminus\Gamma_q$ and obtain a decomposition of $X=\bigcup_{i=1}^{n}X_i$. Then, we can choose a subinterval $\tau_1$ of $\tau$ such that $\tau_1$ avoids containing any endpoints of the vertical edges of the rectangles formed by the first return map on $\tau$, where the vertical edges contain singularities. Therefore, if we consider each critical trajectory starting from a singularity, the first point where the trajectory reaches $\tau$ is not in $\tau_1$. This ensures that, for each critical trajectory satrting from a singularity, the first point at which the trajectory reaches $\tau_1$ lies further along the trajectory, resulting in an increased length for each critical trajectory from a singularity to the first point it hits within $\tau_1$. 

If we glue the cylinders and rectangles formed by the first return map on $\tau_1$ along their edges contain singularities and saddle connections, another decomposition of $X$ is obtained. Similarly, the surface $X_t$ also admits an analogous decomposition related to $\tau_1$. Then we can choose a sufficiently large $t_1>0$ such that, for each rectangle on $X_{t_1}$ formed by the first return map on $\tau_1$, the width of the rectangle exceeds its height. Thus, under the normalized singular flat metrics, the subsurface $X_i$ in the decomposition of $X$ associated with $\tau$ can be isometrically embedded into the corresponding subsurface $X_{t_1,i}$ in the decomposition of $X_t$ associated with $\tau_1$, where the embedding  preserves the edges of the rectangles and cylinders along the critical graph $\widetilde{\Gamma}_q$.

We then choose a subinterval $\tau_2$ of $\tau_1$ in the same way as selecting $\tau_1$ from $\tau$ and a sufficiently large $t_2>t_1$. The widths of the rectangles on $X_{t_2}$ formed by the first return map on $\tau_2$ exceed their heights. Similarly, under the normalized singular flat metrics, the subsurface $X_{t_1,i}$ in the decomposition of $X_{t_1}$ associated with $\tau_1$ can be isometrically embedded into the corresponding subsurface $X_{t_2,i}$ in the decomposition of $X_{t_2}$ associated with $\tau_2$, where the embedding preserves the edges of the rectangles and cylinders along the critical graph $\widetilde{\Gamma}_q$. Therefore, by repeatedly applying this process, we can obtain a sequence $X_{t_k,i}$ along the Teichm\"uller ray $\mathcal{R}_{q,X}(t)$ such that each surface $X_{t_k,i}$ can be isometrically embedded into the surface $X_{t_{k+1},i}$ preserving the edges of the rectangles and cylinders along the critical graph $\widetilde{\Gamma}_q$.

\subsection{The half-plane surfaces}

Consider a finite connected metric graph $G$ which satisfies that:
\begin{itemize}
    \item[(1)] the metric graph $G$ allows loops and multiple edges;
    \item[(2)] the edges with a vertex of degree $1$ are allowed to be of infinite length, while other edges are of finite length.
\end{itemize}
Such a metric graph $G$ is called an admissible metric graph if $G$ satisfies these conditions.

The half plane is the upper half Euclidean plane with boundary $\mathbb{R}$ in $\mathbb{C}$, and the semi-infinite cylinder is a Euclidean cylinder $S^1\times\mathbb{R}_{\geq 0}$ which is holomorphic to $\overline{\mathbb{D}}^\ast=\{z\in\mathbb{C}\mid 0<z\leq 1\}$. Given an admissible metric graph $G$, we can glue half planes and semi-infinite cylinders along the edges of $G$ by isometries on the boundaries. If $G$ has no infinite length edges, we can only glue semi-infinite cylinders along the edges of $G$. This construction forms a surface such that the admissible metric graph $G$ is isometrically embedded in the surface.

\begin{example}
    The graph $G$ in Figure \ref{example} consists of five vertices and six edges, where the edges $a$ and $f$ have infinite length. By gluing two half planes and two semi-infinite cylinders along the edges of $G$, we obtain a surface which is homotopic to a sphere with three punctures.
\end{example}

\begin{figure}[!htpb]
    \centering
    \fontsize{7pt}{9pt}\selectfont
    \def\svgwidth{1\columnwidth}
    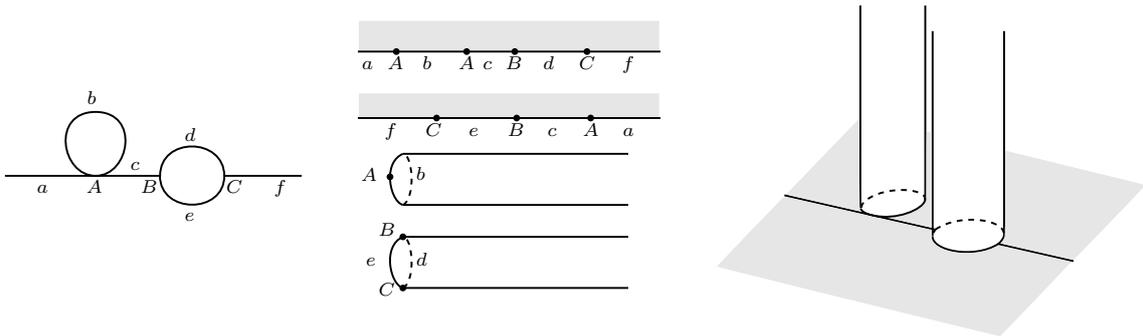

    \caption{The half-plane surface is formed by gluing two half planes and two semi-infinite cylinders along the edges of the admissible metric graph.}
    \label{example}
\end{figure}

\begin{definition}
    Given an admissible metric graph $G$, if the surface obtained by gluing half planes and semi-infinite cylinders along the edges of $G$ by isometries on the boundaries is orientable, the surface is called a half-plane surface.
\end{definition}

\begin{figure}[!htpb]
    \centering
    \fontsize{7pt}{9pt}\selectfont
    \def\svgwidth{0.8\columnwidth}
    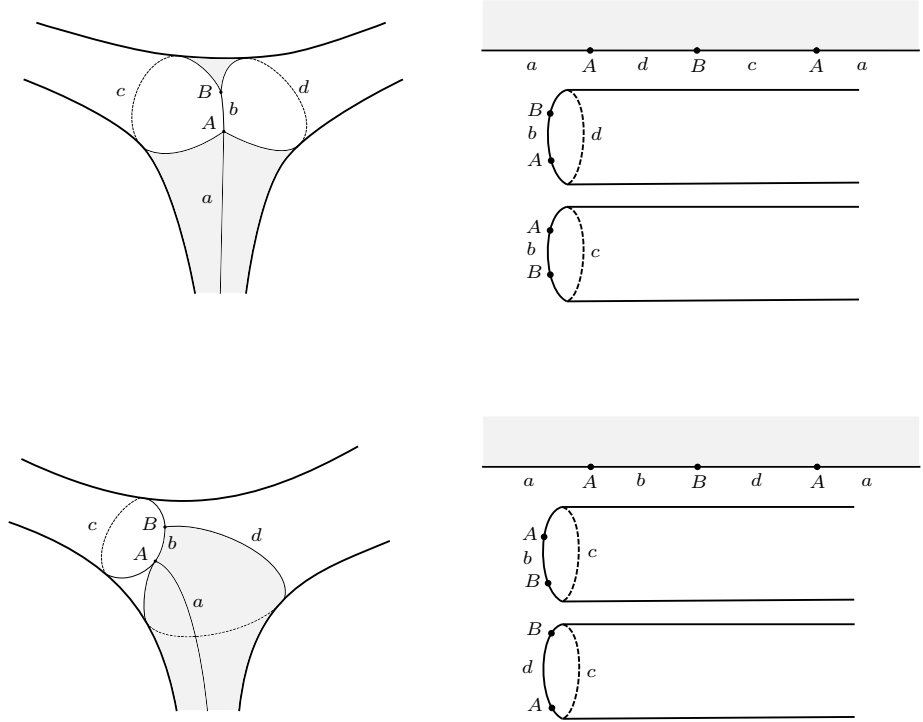

    \caption{The two half-plane surfaces with the same admissible metric graph are obtained by gluing one half plane and two semi-infinite cylinders.}
    \label{metricgraph}
\end{figure}

Note that there can be multiple half-plane surfaces associated with an admissible metric graph $G$, which implies that $G$ can be embedded isometrically in different half-plane surfaces (see Figure \ref{metricgraph}). The half-plane surface can be endowed with a meromorphic quadratic differential $q$, which is represented as $dz^2$ on each half plane and $\frac{dz^2}{z^2}$ locally on $\overline{\mathbb{D}}^\ast$ for each semi-infinite cylinder. This meromorphic quadratic differential $q$ defines a singular flat structure on the half-plane surface, which uniquely extends across the singularities of $q$ to induce a complex structure on the entire surface. Therefore, the half-plane surface is conformally equivalent to a Riemann surface $X^\ast$ endowed with a meromorphic quadratic differential, and there are finitely many poles of order $n\geq 2$ at the punctures of $X^\ast$ formed by the half planes and semi-infinite cylinders. A pole of order $2$ is formed by a semi-infinite cylinder, and a pole of order $n>2$ is formed by $n-2$ half planes. 

Strebel \cite{Str1984} proved the existence of a meromorphic quadratic differential with poles of order $2$ for a Riemann surface, given prescribed local data. There is a singular flat metric on the Riemann surface induced by the meromorphic quadratic differential with poles of order $2$ such that the surface consists of a collection of semi-infinite cylinders glued by isometries on their boundaries. Gupta extended Strebel's result to the case of meromorphic quadratic differential with higher-order poles (see \cite{Gup2014} and \cite{GW2016}). The Riemann surface under the associated singular flat metric, induced by the meromorphic quadratic differential with higher-order poles, is isometric to a collection of half-planes glued by an interval exchange mapping on their boundaries.

The critical graph $\widetilde{\Gamma}_q$ on $X$ admits a decomposition analogous to that of the finite critical graph $\Gamma_q$. Specifically, $\widetilde{\Gamma}_q$ can be written as
$$
\widetilde{\Gamma}_q=\widetilde{\Gamma}_{q,1}\cup\widetilde{\Gamma}_{q,2}\cup\cdots\cup\widetilde{\Gamma}_{q,n},
$$
where each $\widetilde{\Gamma}_{q,i}$ is a subgraph of $\widetilde{\Gamma}_q$ containing the corresponding subgraph $\Gamma_{q,i}$ of $\Gamma_q$. It is clear that each $\widetilde{\Gamma}_{q,i}$ is an admissible metric graph. For the decomposition $X=\bigcup_{i=1}^nX_i$ of the Riemann surface $X$ described in \S\ref{TheDecomposition}, each surface $X_i$ corresponds to a subgraph $\widetilde{\Gamma}_{q,i}$ of $\widetilde{\Gamma}_q$. By gluing half planes and semi-infinite cylinders along the edges of $\widetilde{\Gamma}_{q,i}$ in a manner analogous to the gluing of cylinders and rectangles in \S\ref{TheDecomposition}, we obtain a half-plane surface $X_{\infty,i}$.  The surface $X_i$ can be isometrically embedded into $X_{\infty,i}$ in a way that preserves the edges of the rectangles and cylinders along the graph $\widetilde{\Gamma}_{q,i}$. Similarly, for a surface $X_t=\bigcup_{i=1}^nX_{t,i}$ along the Teichm\"uller ray $\mathcal{R}_{q,X}(t)$, each subsurface $X_{t,i}$ can also be isometrically embedded in $X_{\infty,i}$. Following the construction in \S\ref{TheDecomposition}, we can obtain a sequence $X_{t_k,i}$ along the Teichm\"uller ray $\mathcal{R}_{q,X}(t)$, where each surface $X_{t_k,i}$ is isometrically embedded into $X_{t_{k+1},i}$. This sequence forms an exhaustion of the surface $X_{\infty,i}$.

Gupta also discussed the half plane surface associated with a Teichm\"uller ray, referred to as the conformal limit of the Teichm\"uller ray, in \cite{Gup2019}. Furthermore, Gupta showed that there exists a harmonic map from the conformal limit of a Teichm\"uller ray to a crowned hyperbolic surface. In this paper, we focus on the convergence of surface along a Teichm\"uller ray to its conformal limit and define the distance between the conformal limits of two Teichm\"uller rays.

\subsection{The limit surfaces}\label{limit surfaces}

We recall the Gromov-Hausdorff convergence for sequences of metric spaces (see \cite{BH1999}). An \emph{$\varepsilon$-relation} between two metric spaces $\varSigma_1$ and $\varSigma_2$ is a subset $\Lambda\subseteq \varSigma_1\times \varSigma_2$ such that:
\begin{itemize}
    \item[(1)] the projections of $\Lambda$ onto $\varSigma_1$ and $\varSigma_2$ respectively are surjective;
    \item[(2)] if $(x_1,y_1),(x_2,y_2)\in\Lambda$ then $\left|d_{\varSigma_1}(x_1,x_2)-d_{\varSigma_2}(y_1,y_2)\right|<\varepsilon$, where $d_{\varSigma_1}$ and $d_{\varSigma_2}$ are metrics on $\varSigma_1$ and $\varSigma_2$ respectively. 
\end{itemize}
We denote by $\varSigma_1\simeq_\varepsilon\varSigma_2$ if there is an $\varepsilon$-relation between $\varSigma_1$ and $\varSigma_2$, and we denote by $x\Lambda y$ if $(x,y)\in\Lambda$. The \emph{Gromov-Hausdorff distance} between $\varSigma_1$ and $\varSigma_2$ is defined as
$$
d_{GH}(\varSigma_1,\varSigma_2):=\inf\{\varepsilon\mid\varSigma_1\simeq_\varepsilon\varSigma_2\}.
$$
We say that a sequence of metric spaces $\varSigma_n$ converges to $\varSigma$ in the Gromov-Hausdorff sense if and only if $d_{GH}(\varSigma_n,\varSigma)\to 0$ as $n\to\infty$.

For the convergence of non-compact metric spaces, we consider the metric space $\varSigma$ with a basepoint $x\in\varSigma$. A sequence of pointed metric space $(\varSigma_n,x_n)$ is said to converge to $(\varSigma,x)$ if for any $r>0$, the sequence of closed balls $\overline{B}(x_n,r)\subseteq\varSigma_n$ converges to $\overline{B}(x,r)\subseteq\varSigma$ in the Gromov-Hausdorff sense. Then we call that $(\varSigma_n,x_n)$ converges to $(\varSigma,x)$ in the pointed Gromov-Hausdorff sense.

Let $x_\ast$ be a singularity of the vertical measured foliation $V(q)$ on $X\in\mathcal{T}(S)$. Since the finite critical graph $\Gamma_q$ is preserved along the Teichm\"uller ray $\mathcal{R}_{q,X}(t)$, We can consider the convergence of the sequence $(X_t,x_\ast)$ with the singular flat metric induced by $e^{2t}q_t$ along the ray $\mathcal{R}_{q,X}(t)$ in the sense of pointed Gromov-Hausdorff.

\begin{lemma}\label{TheGromovHausdorffConvergence}
    Let $\Gamma_q=\bigcup_{i=1}^n\Gamma_{q,i}$ be the finite critical graph of $q$ on $X$ and $X_t=\bigcup_{i=1}^n X_{t,i}$ be the surface with the singular flat metric induced by $e^{2t}q_t$ along the Teichm\"uller ray $\mathcal{R}_{q,X}(t)$. If $x_i$ is a singularity in $\Gamma_{q,i}$, then the sequence $(X_t,x_i)$ converges to the half-plane surface $(X_{\infty,i},x_i)$ in the pointed Gromov-Hausdorff sense.
\end{lemma}

\begin{proof}
    For any $r>0$, let $\overline{B}_t(x_i,r)$ be a closed ball in $X_t$ and $\overline{B}_\infty(x_i,r)$ be a closed ball in $X_{\infty,i}$. We can pick an appropriate horizontal segment $\tau_t$ for each minimal component of $X_t\setminus\Gamma_q$ and a sufficiently large $t$ such that the subsurface $X_{t,i}$ in the decomposition of $X_t$ associated with $\tau_t$ contains the closed ball $\overline{B}_t(x_i,r)$. Since the subsurface $X_{t,i}$ can be isometrically embedded in the surface $X_{\infty,i}$ preserving the graph $\Gamma_{q,i}$, this implies that $\overline{B}_t(x_i,r)$ converges to $\overline{B}_\infty(x_i,r)$ in the Gromov-Hausdorff sense. Then the sequence $(X_t,x_i)$ converges to the surface $(X_{\infty,i},x_i)$ in the sense of pointed Gromov-Hausdorff.
\end{proof}

\begin{xrem}
    From the proof of Lemma \ref{TheGromovHausdorffConvergence}, we can pick the horizontal segment $\tau$ for each minimal component of $X\setminus\Gamma_q$ such that for sufficiently large $t$, the subsurface $X_{t,i}$ contains the closed ball $\overline{B}_t(x_i,r)$. This implies that $(X_{t,i},x_i)$ converges to the surface $(X_{\infty,i},x_i)$ in the pointed Gromov-Hausdorff sense. Since there is an isometric embedding from $X_{t,i}$ to $X_{\infty,i}$ preserving the graph $\Gamma_{q,i}$, we can treat $X_{t,i}$ as a subsurface of $X_{\infty,i}$. As described in \S\ref{TheDecomposition}, by selecting an appropriate horizontal segment $\tau_t$ for each minimal component of $X_t\setminus\Gamma_q$, we can obtain a sequence of surfaces $X_{t,i}$ along the Teichm\"uller ray $\mathcal{R}_{q,X}(t)$, which forms an exhaustion of the surface $X_{\infty,i}$. Since $X_{t,i}$ contains $\Gamma_{q,i}$, and from the decomposition of $X_t$, each boundary component of $X_{t,i}$ is a simple closed curve composed of vertical segments along the leaves of $V(q)$ and horizontal segments on $\tau_t$, we have the following equality for the Euler characteristic:
    $$
    \chi(\Gamma_{q,i})=\chi(X_{t,i})=\chi(X_{\infty,i})=2-2g_i-n_i
    $$
    for any $i=1,\cdots,n$, where $g_i$ is the genus of the surface $X_{\infty,i}$ and $n_i$ is the number of punctures on $X_{\infty,i}$.
\end{xrem}

\begin{proposition}\label{TheGHconvergenceWithDifferentBasepoints}
    Let $(X_t,x_t)$ be a sequence of Riemann surface $X_t=\bigcup_{i=1}^n X_{t,i}$ with basepoint $x_t\in X_t$ along the Teichm\"uller ray $\mathcal{R}_{q,X}(t)$. If for a singularity $x_i\in\Gamma_{q,i}$, the distance between $x_i$ and $x_t$ is uniformly bounded on each $X_t$, then there is a subsequence of $(X_t,x_t)$ converging to the half-plane surface $(X_{\infty,i},x_\infty)$ in the sense of pointed Gromov-Hausdorff, where $x_\infty$ is a point in $X_{\infty,i}$. 
\end{proposition}

\begin{proof}
    Let $d_t$ be the singular flat metric on $X_t$ and $d_\infty$ be the singular flat metric on $X_{\infty,i}$. Since the distance between $x_i$ and $x_t$ is uniformly bounded on each $X_t$, there exists a constant $M>0$ such that $d_t(x_i,x_t)<M$ for any $t\geq 0$. We can choose an appropriate horizontal segment $\tau_t$ for each minimal component of $X_t\setminus\Gamma_q$ such that for sufficiently large $t$, the subsurface $X_{t,i}$ contains the closed ball $\overline{B}_t(x_i,M)\subset X_t$. Since $X_{t,i}$ can be isometrically embedded in the surface $X_{\infty,i}$ and $x_t\in\overline{B}_t(x_i,M)$, we can regard $x_t$ as a point in $X_{\infty,i}$ for sufficiently large $t$. Then $x_t$ is in the closed ball $\overline{B}_\infty(x_i,M)\subset X_{\infty,i}$. There is a subsequence of $\{x_t\}$ converging to a point $x_\infty\in \overline{B}_\infty(x_i,M)$. We show that the subsequence of $(X_t,x_t)$ corresponding to the subsequence of $\{x_t\}$ converges to $(X_{\infty,i},x_\infty)$ in the sense of pointed Gromov-Hausdorff. We still denote by $(X_t,x_t)$ the subsequence of $(X_t,x_t)$ for simplicity. 

    For any $r>0$, we need to show that $\overline{B}_t(x_t,r)$ converges to $\overline{B}_\infty(x_\infty,r)$ in the Gromov-Hausdorff sense. Similarly, we choose an appropriate small horizontal segment $\tau_t$ for each minimal component of $X_t\setminus\Gamma_q$ and sufficiently large $t$ such that $\overline{B}_t(x_t,r)\subset X_{t,i}$. Since $X_{t,i}$ can be isometrically embedded in the surface $X_{\infty,i}$, we can regard $\overline{B}_t(x_t,r)$ as a closed ball in $X_{\infty,i}$, and for any $x_1,x_2\in\overline{B}_t(x_t,r)$, the distance $d_t(x_1,x_2)=d_\infty(x_1,x_2)$. For any $\varepsilon>0$, we have $d_\infty(x_t,x_\infty)<\frac{\varepsilon}{3}$ for sufficiently large $t$. Let
    $$
    \Lambda_t=\left\{(x,y)\in\overline{B}_t(x_t,r)\times\overline{B}_\infty(x_\infty,r)\mid d_\infty(x,y)<\frac{\varepsilon}{2}\right\}.
    $$
    For any $x\in\overline{B}_t(x_t,r)$, consider a neighborhood $U(x,\frac{\varepsilon}{2})=\{y\in X_{\infty,i}\mid d_\infty(x,y)<\frac{\varepsilon}{2}\}$ of $x$. Since $d_\infty(x_t,x_\infty)<\frac{\varepsilon}{3}$ for sufficiently large $t$, It is easy to know that $\{U(x,\frac{\varepsilon}{2})\mid x\in\overline{B}_t(x_t,r)\}$ can cover $\overline{B}_\infty(x_\infty,r)$. This implies that the projections of $\Lambda_t$ onto $\overline{B}_t(x_t,r)$ and $\overline{B}_\infty(x_\infty,r)$ respectively are surjective. For any $(x_1,y_1),(x_2,y_2)\in\Lambda_t$, we have 
    $$
    |d_t(x_1,x_2)-d_\infty(y_1,y_2)|=|d_\infty(x_1,x_2)-d_\infty(y_1,y_2)|\leq d_\infty(x_1,y_1)+d_\infty(y_2,x_2)<\varepsilon.
    $$
    Then $\Lambda_t$ is an $\varepsilon$-relation between $\overline{B}_t(x_t,r)$ and $\overline{B}_\infty(x_\infty,r)$ for sufficiently large $t$. This shows that $\overline{B}_t(x_t,r)$ converges to $\overline{B}_\infty(x_\infty,r)$ in the Gromov-Hausdorff sense.
\end{proof}

We define the limit surface $X_\infty$ of the Teichm\"uller ray $\mathcal{R}_{q,X}(t)$ as
$$
X_\infty=\bigcup_{i=1}^n X_{\infty,i},
$$
where $X_{\infty,i}$ is the half-plane surface converged by the subsurface $X_{t,i}$ of $X_t=\bigcup_{i=1}^n X_{t,i}$ in the sense of pointed Gromov-Hausdorff. 

Let $f:S\to X$ be the marking of $X$ and $f_t:X\to X_t$ be the Teichm\"uller mapping between $X$ and $X_t$. For a decomposition $X_t=\bigcup_{i=1}^n X_{t,i}$ of $X_t$, there is a decomposition of $S$ denoted by 
$$
S=\bigcup_{i=1}^n S_i,
$$
where $S_i=f^{-1}\circ f_t^{-1}(X_{t,i})$. Similarly, if we consider the graph $f^{-1}(\widetilde{\Gamma}_q)$ on $S$, for each subgraph $\widetilde{\Gamma}_{q,i}$ of $\widetilde{\Gamma}_q$, there is a half-plane surface $S_{\infty,i}$ obtained by gluing half planes and semi-infinite cylinders along the edges of $f^{-1}(\widetilde{\Gamma}_{q,i})$ as the pattern of $S_i$. Since the mapping $f:S\to X$ is a quasiconformal mapping, for each $X_i\subset X$, we consider the restriction of $f$ on $f^{-1}(X_i)\subset S$ and extend $f\big|_{f^{-1}(X_i)}$ to a quasiconformal mapping $g_i:S_{\infty,i}\to X_{\infty,i}$ up to homotopy. Then the surface $X_{\infty,i}=[X_{\infty,i},g_i]$ is in the Teichm\"uller space $\mathcal{T}(S_{\infty,i})$. 

We define the marked limit surface for the Teichm\"uller ray $\mathcal{R}_{q,X}(t)$ as
$$
X_\infty=\bigcup_{i=1}^n X_{\infty,i}=\bigcup_{i=1}^n [X_{\infty,i},g_i].
$$
It is clear that 
$$
X_\infty\in\prod_{i=1}^n \mathcal{T}(S_{\infty,i}).
$$
Then we define the Teichm\"uller distance $d_{\overline{\mathcal{T}}}$ between two marked limit surfaces $X_\infty$ and $Y_\infty$ as follows:
\begin{itemize}
    \item If $X_\infty,Y_\infty\in\prod_{i=1}^n \mathcal{T}(S_{\infty,i})$, the Teichm\"uller distance is 
    $$
    d_{\overline{\mathcal{T}}}(X_\infty,Y_\infty)=\max_{1\leq i\leq n}\{d_{\mathcal{T}_i}(X_{\infty,i},Y_{\infty,i})\},
    $$
    where $d_{\mathcal{T}_i}$ is the Teichm\"uller metric on $\mathcal{T}(S_{\infty,i})$;
    \item Otherwise, the Teichm\"uller distance is
    $$
    d_{\overline{\mathcal{T}}}(X_\infty,Y_\infty)=+\infty.
    $$
\end{itemize}

\section{Upper estimate of the limiting Teichm\"uller distance}

Let $\mathcal{R}_{q,X}(t)$ and $\mathcal{R}_{q^\prime,Y}(t)$ be two Teichm\"uller rays. The vertical measured foliations $V(q)=\sum_{j=1}^{N}a_jG_j$ and $V(q^\prime)=\sum_{j=1}^{N}b_jG_j$ are absolutely continuous. In this section, we assume that each $G_j$ is a simple closed curve or a uniquely ergodic measure.

Let $X=\bigcup_{i=1}^{n}X_i$ be the decomposition of $X$ as in \S\ref{TheDecomposition}. We give a decomposition of $Y$ which is similar to the decomposition of $X$. Since $V(q)$ and $V(q^\prime)$ are absolutely continuous, for a minimal component $\Omega$ of $X\setminus\Gamma_q$, there is a corresponding minimal component $\Omega^\prime$ of $Y\setminus\Gamma_{q^\prime}$. We assume that the $(G_j,a_j\mu_j)$ and $(G_j,b_j\mu_j)$ are the restrictions of $V(q)$ and $V(q^\prime)$ on $\Omega$ and $\Omega^\prime$ respectively, where $\mu_j$ is the uniquely ergodic measure.

For a horizontal segment $\tau$, we have a rectangular decomposition of $\Omega$, that is
$$
\Omega=R_1\cup R_2\cup\cdots\cup R_m,
$$
Since $V(q)$ and $V(q^\prime)$ are topologically equivalent, there is a homeomorphism $h:X\setminus\Gamma_q\to Y\setminus\Gamma_{q^\prime}$ that takes the leaves of $V(q)$ to the leaves of $V(q^\prime)$. Let $\gamma_L$ and $\gamma_R$ be the two leaves of $V(q)$ that each contains an endpoint of $\tau$. We choose a horizontal segment $\tau^\prime$ between $h(\gamma_L)$ and $h(\gamma_R)$ on $Y$ such that $\tau^\prime$ is isotopic to $h(\tau)$, and the first return mappings on $\tau$ and $\tau^\prime$ are identical. Then we have 
$$
\frac{\ell(\tau^\prime)}{\ell(\tau)}=\frac{b_j\mu_j(\tau^\prime)}{a_j\mu_j(\tau)}=\frac{b_j}{a_j},
$$
where $\ell(\tau)$ and $\ell(\tau^\prime)$ are the lengths of $\tau$ and $\tau^\prime$ respectively. Then we can obtain a rectangular decomposition of $\Omega^\prime$, that is 
$$
\Omega^\prime=R_1^\prime\cup R_2^\prime\cup\cdots\cup R_m^\prime.
$$
The width $\ell(R_i)$ of $R_i$ and the width $\ell(R_i^\prime)$ of $R_i^\prime$ satisfy 
$$
\frac{\ell(R_i^\prime)}{\ell(R_i)}=\frac{b_j\mu_j(R_i^\prime)}{a_j\mu_j(R_i)}=\frac{b_j}{a_j},\quad \text{for any }1\leq i\leq k.
$$

Let $A^\prime$ be the cylindrical component of $Y\setminus\Gamma_{q^\prime}$ corresponding to the cylinder $A$ of $X\setminus\Gamma_q$. The ratio of the heights of $A^\prime$ and $A$ is $\frac{b_j}{a_j}$. Similarly, we split $A^\prime$ into two cylinders with the same height along a closed leaf of $V(q^\prime)$. Then the surface $Y$ is the union of the cylinders and rectangles, and if we glue the cylinders and rectangles along their boundaries containing singularities and saddle connections, we obtain a similar decomposition of $Y$ that is 
$$
Y=Y_1\cup Y_2\cup\cdots\cup Y_n,
$$
where $Y_i$ is the subsurface having the same type as the subsurface $X_i$ of $X$. It is the same for $Y_t$ along $\mathcal{R}_{q^\prime,Y}(t)$, which has the similar decomposition to $X_t$, and can be written as $Y_t=\bigcup_{i=1}^{n}Y_{t,i}$. The finite critical graph $\Gamma_{q^\prime}$ can be written as $\Gamma_{q^\prime}=\bigcup_{i=1}^{n}\Gamma_{q^\prime,i}$, where $\Gamma_{q^\prime,i}$ is a connected subgraph of $\Gamma_{q^\prime}$ which is contained in the subsurface $Y_{q^\prime,i}$.

Motivated by Masur's method in \cite{Mas1980}, we generalize this method to the more general case and obtain the following lemma.

\begin{lemma}\label{TheGluingMapping}
    Let $X_t$ be a surface along $\mathcal{R}_{q,X}(t)$ and $Y_t$ be a surface along $\mathcal{R}_{q^\prime,Y}(t)$. The vertical measured foliations $V(q)=\sum_{j=1}^{N}a_jG_j$ and $V(q^\prime)=\sum_{j=1}^{N}b_jG_j$ are absolutely continuous, where each $G_j$ is a simple closed curve or a uniquely ergodic measure. Then for any $\varepsilon>0$, there exist decompositions $X_t=\bigcup_{i=1}^{n}X_{t,i}$ and $Y_t=\bigcup_{i=1}^{n}Y_{t,i}$ such that for sufficiently large $t$ and any $1\leq i\leq n$, there is a quasiconformal mapping $g_{t,i}:X_{t,i}\setminus\overline{\Gamma}_{q,i}\to Y_{t,i}\setminus\overline{\Gamma}_{q^\prime,i}$ with the dilatation 
    $$
    K(g_{t,i})\leq\max_{1\leq j\leq N}\left\{\frac{m_j^\prime}{m_j},\frac{m_j}{m_j^\prime}\right\}+O(\varepsilon),
    $$
    where $\overline{\Gamma}_{q,i}$ $(\overline{\Gamma}_{q^\prime,i})$ is a compact subset of $\widetilde{\Gamma}_q$ $(\widetilde{\Gamma}_{q^\prime})$.
\end{lemma}

\begin{proof}
    Pick a small horizontal segment $\tau$ on a minimal component $\Omega$ of $X\setminus\Gamma_q$. The minimal component $\Omega$ has a rectangular decomposition $\Omega=\bigcup_{i=1}^{m}R_i$. Since $V(q)$ and $V(q^\prime)$ are absolutely continuous, we can pick a horizontal segment $\tau^\prime$ on the minimal component $\Omega^\prime$ of $Y\setminus\Gamma_{q^\prime}$ such that the first return mappings on $\tau$ and $\tau^\prime$ are identical. The rectangular decomposition of $\Omega^\prime$ is $\Omega^\prime=\bigcup_{i=1}^{m}R_i^\prime$.

    Let $(G_j,a_j\mu_j)$ be the measured foliation on $\Omega$ and $(G_j,b_j\mu_j)$ be the measured foliation on $\Omega^\prime$. Since the transverse measure $\mu_j$ is uniquely ergodic, the $T$ is uniquely ergodic on $\tau_+\cup\tau_-$. By Birkhoff's ergodic theorem, for any function $f$ on $\tau$, we have 
    $$
    \lim_{n\to\infty}\frac{1}{n}\sum_{k=0}^{n-1}f\circ T^k(x)=\int fd\mu_j.
    $$
    A routine approximation shows the same to be true if we replace $f$ by the characteristic function of an open interval. Then we consider the characteristic function $\chi_{R_i}$ of a horizontal edge of $R_i$. For any $\varepsilon>0$, we pick $N'$ large enough such that for all $n\geq N'$ and any $x\in\tau$,
    \begin{equation}\label{ergodicThm}
        \left|\frac{1}{n}\sum_{k=0}^{n-1}\chi_{R_i}\circ T^k(x)-\mu_j(R_i)\right|<\varepsilon.
    \end{equation}
    The same is true for $R_i^\prime$ and $\tau^\prime$.

    We can pick a subinterval $\sigma\subset\tau$ such that for any $x\in\sigma$, $T^k(x)$ is not a vertex of $R_i$ and $T^k(x)\notin\sigma$ for $0<k\leq N'-1$ and $-N'+1\leq k<0$. For the $\tau^\prime$, we can also pick a subinterval $\sigma^\prime\subset\tau^\prime$ satisfying the same condition. Then, if we consider the first return mappings on $\sigma$ and $\sigma^\prime$, there are similar rectangular decompositions $\Omega=\bigcup_{j=1}^{m}R^\sigma_j$ and $\Omega^\prime=\bigcup_{j=1}^{m}R^{\sigma^\prime}_j$. For a point $x$ on the horizontal edge of $R^\sigma_j$, Let $v_i$ be the number of visits of $x$ to $R_i$ before returning to $\sigma$. This is the same as the number of visits of $x$ to $R_i^\prime$ before returning to $\sigma^\prime$ for $x$ on the horizontal edge of $R^{\sigma^\prime}_j$. The $v_i$ is independent of the choice of $x$. We use $|\cdot|$ to denote the height of a rectangle. Then 
    $$
    |R^\sigma_j|=\sum_{i=1}^{m}|R_i|v_i,\quad |R^{\sigma^\prime}_j|=\sum_{i=1}^{m}|R_i^\prime|v_i.
    $$
    Let $v=\sum_{i=1}^{m}v_i$. Then we have $|\frac{v_i}{v}-\mu_j(R_i)|<\varepsilon$ by \eqref{ergodicThm}, and the same holds true for $R_i^\prime$. Thus,
    \begin{align*}
        & \frac{i(G_j,H(q^\prime))-\varepsilon\sum_{i=1}^{m}|R_i^\prime|}{i(G_j,H(q))+\varepsilon\sum_{i=1}^{m}|R_i|}
        =\frac{\sum_{i=1}^{m}|R_i^\prime|(\mu_j(R_i^\prime)-\varepsilon)}{\sum_{i=1}^{m}|R_i|(\mu_j(R_i)+\varepsilon)}
        \leq\frac{\sum_{i=1}^{m}|R_i^\prime|\frac{v_i}{v}}{\sum_{i=1}^{m}|R_i|\frac{v_i}{v}} \\
        & \leq\frac{\sum_{i=1}^{m}|R_i^\prime|(\mu_j(R_i^\prime)+\varepsilon)}{\sum_{i=1}^{m}|R_i|(\mu_j(R_i)-\varepsilon)}
        =\frac{i(G_j,H(q^\prime))+\varepsilon\sum_{i=1}^{m}|R_i^\prime|}{i(G_j,H(q))-\varepsilon\sum_{i=1}^{m}|R_i|}.
    \end{align*}
    Therefore,
    $$
    \frac{|R^{\sigma^\prime}_j|}{|R^\sigma_j|}=\frac{i(G_j,H(q^\prime))}{i(G_j,H(q))}+O(\varepsilon)\text{ as } \varepsilon\to 0.
    $$
    In all estimates, $O(\varepsilon)$ refers to a quantity such that $O(\varepsilon)\leq C\varepsilon$, where $C>0$ is some constant depending only on the initial surfaces and quadratic differentials.

    We pick a point on each half-infinite critical trajectory of $q$ which is close to the critical endpoint of the trajectory such that for each rectangle $R_i$, all critical points on the vertical edge of $R_i$ are between the two points we picked as in the Figure \ref{pickpoints}. Then we obtain a compact subset $\overline{\Gamma}_q=\bigcup_{i=1}^{n}\overline{\Gamma}_{q,i}$ of $\widetilde{\Gamma}_q$, where $\overline{\Gamma}_{q,i}$ is a connected subgraph of $\overline{\Gamma}_q$. In the same way, there is a graph $\overline{\Gamma}_{q^\prime}=\bigcup_{i=1}^{n}\overline{\Gamma}_{q^\prime,i}$ on $Y$.

    \begin{figure}[!htpb]
        \centering
    \fontsize{7pt}{9pt}\selectfont
    \def\svgwidth{0.9\columnwidth}
    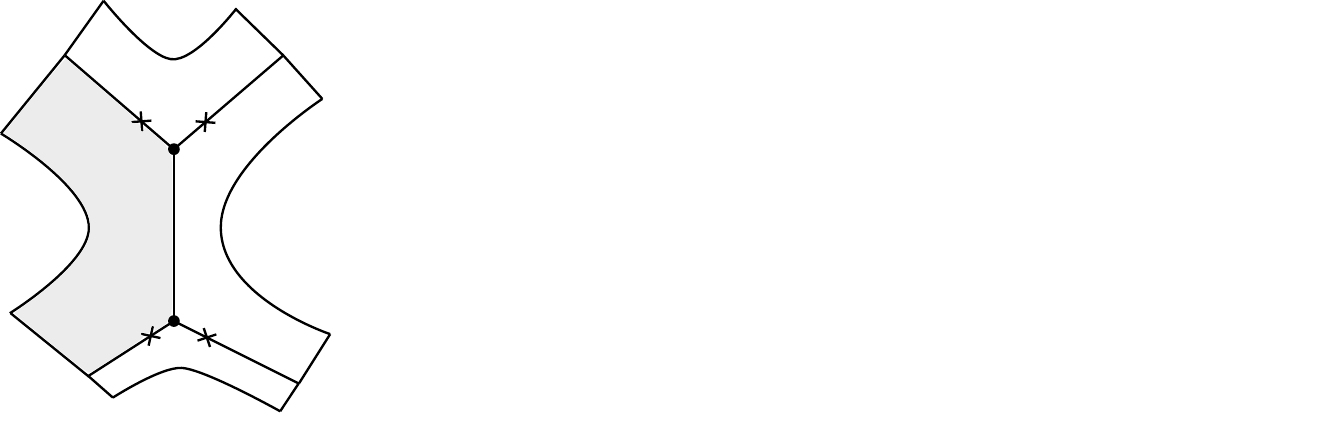

        \caption{Pick a point on each half-infinite critical trajectory such that all singularities are between the picked points for each $R_i$ and $R_i^\prime$.}
        \label{pickpoints}
    \end{figure}

    Consider the two points we picked on the vertical edge of $R^\sigma_j$. Let $\beta$ be the critical segment between one point we picked and a vertex of $R^\sigma_j$, and let $v_i^\prime$ be the number of visits of $\beta$ to $R_i$. For the $R^{\sigma^\prime}_j$, there is the corresponding segment $\beta^\prime$ on the vertical edge, and the number of visits of $\beta^\prime$ to $R_i^\prime$ is equal to $v_i^\prime$. Let $v^\prime=\sum_{i=1}^{m}v_i^\prime$. Since the first return mapping $T$ satisfies the conditions mentioned above for $\sigma$ and $\sigma^\prime$, we still have $|\frac{v_i^\prime}{v^\prime}-\mu_j(R_i)|<\varepsilon$ by \eqref{ergodicThm}. Therefore, we can obtain that the ratio of the lengths of $\beta$ and $\beta^\prime$ is
    $$
    \frac{|\beta^\prime|}{|\beta|}=\frac{i(G_j,H(q^\prime))}{i(G_j,H(q))}+O(\varepsilon)\text{ as } \varepsilon\to 0.
    $$

    \begin{lemma}\label{AffineMapping}
        Let $R$ and $R'$ be two rectangles with vertices $A_i$, $A_i'$ $(i=1,\cdots,4.)$ such that the ratio of heights is $\frac{|R'|}{|R|}=\frac{|A_2'A_3'|}{|A_2A_3|}=B+O(\varepsilon)$ as $\varepsilon\to 0$ and the ratio of widths is $\frac{|A_1'A_2'|}{|A_1A_2|}=C$, where $B>0$ and $C>0$ are some constants. The ratio of height and width satisfies $\frac{|R|}{|A_1A_2|}\leq 1$. Suppose there are two points $P_1,P_2$ on the edge $(A_1A_4)$ and the segments $(A_1P_1)$ and $(P_2A_4)$ have no intersection. Similarly, there are two disjoint segments $(A_1'P_1')$ and $(P_2'A_4')$ on the edge $(A_1'A_4')$ such that $\frac{|A_1'P_1'|}{|A_1P_1|}=B+O(\varepsilon)$ and $\frac{|P_2'A_4'|}{|P_2A_4|}=B+O(\varepsilon)$. Then there is a quasiconformal mapping $g:R\to R'$ with the dilatation $K(g)\leq\max\left\{\frac{C}{B},\frac{B}{C}\right\}+O(\varepsilon)$, which is linear on all sides and sends $P_1$ to $P_1'$, $P_2$ to $P_2'$.
    \end{lemma}

    \begin{proof}[\text{Proof of Lemma} \ref{AffineMapping}]
        Let the $A_i$ have coordinates $(0,0)$, $(a,0)$, $(a,b)$, $(0,b)$ in the $z=x+iy$ plane, and $A_i'$ have coordinates $(0,0)$, $(a',0)$, $(a',b')$, $(0,b')$ in the $w=u+iv$ plane. The $P_1$ and $P_2$ have coordinates $(0,c)$, $(0,d)$, and $P_1'$ and $P_2'$ have coordinates $(0,c')$, $(0,d')$. Then we have $\frac{b'}{b}=B+O(\varepsilon)$, $\frac{c'}{c}=B+O(\varepsilon)$, $\frac{b'-d'}{b-d}=B+O(\varepsilon)$ and $\frac{a'}{a}=C$. It is easy to check that $\frac{d'-c'}{d-c}=B+O(\varepsilon)$. Then we can construct the quasiconformal mapping $g:R\to R'$ that is  
        \begin{equation*}
            \begin{aligned}
                u&=\frac{a'}{a}x, & &v=y\left[\left(\frac{b'}{b}-\frac{c'}{c}\right)\frac{x}{a}+\frac{c'}{c}\right], & 0\leq y\leq c; \\
                u&=\frac{a'}{a}x, & &
                \begin{split}
                    v= & y\left[\left(\frac{b'}{b}-\frac{d'-c'}{d-c}\right)\frac{x}{a}+\frac{d'-c'}{d-c}\right] \\
                     & +c\left[\left(\frac{d'-c'}{d-c}-\frac{c'}{c}\right)\frac{x}{a}+\frac{c'}{c}-\frac{d'-c'}{d-c}\right]
                \end{split}, & c\leq y\leq d; \\
                u&=\frac{a'}{a}x, & &v=b'+(y-b)\left[\left(\frac{b'}{b}-\frac{b'-d'}{b-d}\right)\frac{x}{a}+\frac{b'-d'}{b-d}\right], & d\leq y\leq b.
            \end{aligned}
        \end{equation*}
        For $0\leq y\leq c$, 
        \begin{equation*}
            \begin{aligned}
                u_x & =\frac{a'}{a}, & u_y&=0; \\
                v_x & =\left(\frac{b'}{b}-\frac{c'}{c}\right)\frac{y}{a}=O(\varepsilon), & v_y&=\left(\frac{b'}{b}-\frac{c'}{c}\right)\frac{x}{a}+\frac{c'}{c}=B+O(\varepsilon).
            \end{aligned}
        \end{equation*}        
        We can get similar estimates for $c\leq y\leq d$ and $d\leq y\leq b$. Then the mapping $g$ is a quasiconformal mapping with the dilatation $K(g)\leq\max\left\{\frac{C}{B},\frac{B}{C}\right\}+O(\varepsilon)$. By the construction of $g$, the mapping $g$ is affine on all sides and sends $P_1$, $P_2$ to $P_1'$, $P_2'$ respectively.
    \end{proof}
    
    Since the surface $X_t$ and $Y_t$ preserve the vertical leaves of $V(q)$ and $V(q^\prime)$ along the Teichm\"uller rays, there are the rectangular decompositions $\Omega_t=\bigcup_{j=1}^{m}R^{\sigma}_{t,j}$ and $\Omega^\prime_t=\bigcup_{j=1}^{m}R^{\sigma^\prime}_{t,j}$ for the minimal components on $X_t$ and $Y_t$ respectively. Then we still have 
    $$
    \frac{|R^{\sigma^\prime}_{t,j}|}{|R^\sigma_{t,j}|}=\frac{i(G_j,H(q^\prime))}{i(G_j,H(q))}+O(\varepsilon)\text{ as } \varepsilon\to 0.
    $$
    The ratio of the widths $\ell(R^{\sigma^\prime}_{t,j})$ and $\ell(R^\sigma_{t,j})$ is 
    $$
    \frac{\ell(R^{\sigma^\prime}_{t,j})}{\ell(R^\sigma_{t,j})}=\frac{e^{2t}b_j\mu_j(R^{\sigma^\prime}_{t,j})}{e^{2t}a_j\mu_j(R^\sigma_{t,j})}=\frac{b_j}{a_j}.
    $$
    Let $t$ be sufficiently large such that $|R^\sigma_{t,j}|\leq\ell(R^\sigma_{t,j})$. By Lemma \ref{AffineMapping}, there is a quasiconformal mapping $f_{t,j}:R^\sigma_{t,j}\to R^{\sigma^\prime}_{t,j}$, and the dilatation of $f_{t,j}$ satisfies
    $$
    K(f_{t,j})\leq\max\left\{\frac{b_ji(G_j,H(q))}{a_ji(G_j,H(q^\prime))},\frac{a_ji(G_j,H(q^\prime))}{b_ji(G_j,H(q))}\right\}+O(\varepsilon)=\max\left\{\frac{m_j^\prime}{m_j},\frac{m_j}{m_j^\prime}\right\}+O(\varepsilon).
    $$
    Since $f_{t,j}$ is linear on the boundaries of $R^\sigma_{t,j}$, the mappings $f_{t,j}$ can agree along the boundaries of each $R^\sigma_{t,j}$ except for the portion on $\overline{\Gamma}_q$.

    Let $\bar{A}$ be a cylinder on $X$ obtained by splitting a cylindrical component $A$ of $X\setminus\Gamma_q$ into two cylinders of equal height, and let $\bar{A}^\prime$ be the corresponding cylinder on $Y$. Assume that $G_j$ is the measured foliation on $A$ with the modulus $m_j$. Then the modulus on $\bar{A}$ is $\frac{1}{2}m_j$. Similarly, the modulus on $\bar{A}^\prime$ is $\frac{1}{2}m_j^\prime$. We consider the corresponding cylinders $\bar{A}_t$ and $\bar{A}^\prime_t$ on $X_t$ and $Y_t$ respectively. The $\bar{A}_t$ and $\bar{A}^\prime_t$ can be represented by the annuli $\bar{A}_t=\{z\in\mathbb{C}\mid e^{-e^{2t}\pi m_j}\leq z\leq 1\}$ and $\bar{A}^\prime_t=\{w\in\mathbb{C}\mid e^{-e^{2t}\pi m_j^\prime}\leq w\leq 1\}$. Then we can construct a quasiconformal mapping $f_{t,j}:\bar{A}_t\to\bar{A}^\prime_t$ that is 
    $$
    f_{t,j}(z)=|z|^{\frac{m_j^\prime}{m_j}-1}z.
    $$
    The dilatation of $f_{t,j}$ is $K(f_{t,j})=\max\left\{\frac{m_j^\prime}{m_j},\frac{m_j}{m_j^\prime}\right\}$.

    For sufficiently large $t$, there is a quasiconformal mapping $f_{t,j}(z)$ for each of the cylinders and rectangles on $X_t$. By Lemma \ref{AffineMapping} and the construction of the mapping for $\bar{A}_t$, The quasiconformal mappings agree on the boundaries of the cylinders and rectangles except for the the portions on $\overline{\Gamma}_q$. Then we obtain a quasiconformal mapping $g_{t,i}:X_{t,i}\setminus\overline{\Gamma}_{q,i}\to Y_{t,i}\setminus\overline{\Gamma}_{q^\prime,i}$ for any $1\leq i\leq n$, and the dilatation is 
    $$
    K(g_{t,i})\leq\max_{1\leq j\leq N}\left\{\frac{m_j^\prime}{m_j},\frac{m_j}{m_j^\prime}\right\}+O(\varepsilon),
    $$
    Actually, we also get a quasiconformal mapping $g_t:X_t\setminus\overline{\Gamma}_q\to Y_t\setminus\overline{\Gamma}_{q^\prime}$ with the same dilatation.
\end{proof}

Before giving the upper estimate of the limiting Teichm\"uller distance between $\mathcal{R}_{q,X}(t)$ and $\mathcal{R}_{q^\prime,Y}(t)$, We recall some background about the boundary dilatation, the frame mapping theorem \cite{Str1976} and the main inequality of Reich and Strebel \cite{RS1974}. 

Let $f:X\to Y$ be a quasiconformal mapping between the Riemann surfaces $X$ and $Y$. 
We denote by $[f]$ the set of quasiconformal mappings from $X$ to $Y$ which are homotopic to $f$ modulo the boundary. The extremal dilatation of $[f]$ is defined as 
$$
K_0([f])=\inf\{K(g)\mid g\in[f]\}.
$$
The quasiconformal mapping $f$ is called extremal if $K(f)=K_0([f])$. The boundary dilatation of $f$ is defined as
$$
H^\ast(f)=\inf\{K(f\big|_{X\setminus E})\mid E\text{ is a compact subset of }X\},
$$
and the boundary dilatation of $[f]$ is 
$$
H([f])=\inf\{H^\ast(g)\mid g\in[f]\}.
$$ 
It is obvious that $H([f])\leq K_0([f])$.

We state the Strebel's frame mapping theorem and the main inequality of Reich and Strebel as follows. We refer the reader to \cite{GL2000} for more details.

\begin{theorem}[\cite{Str1976}]\label{FrameMappingThm}
    Let $f:X\to Y$ be a quasiconformal mapping between the Riemann surface $X$ and $Y$. If $H([f])< K_0([f])$, then there is a unique extremal quasiconformal mapping $f_0\in[f]$ with the Beltrami coefficient of the form $\mu_{f_0}=k\frac{\bar{q}}{|q|}$, where 
    $$
    0\leq k=\frac{K_0([f])-1}{K_0([f])+1}<1,
    $$
    and $q$ is a holomorphic quadratic differential on $X$ with $\|q\|=1$.
\end{theorem}

\begin{theorem}[\cite{RS1974}]\label{MainInequality}
    Let $f$ and $g$ be two quasiconformal mappings from a Riemann surface $X$ to a Riemann surface $Y$, which are homotopic modulo the boundary. Then, for any integrable holomorphic quadratic differential $q=q(z)dz^2$, we have
    $$    
    \|q\|\leq\iint_X|q(z)|\frac{\left|1-\mu_f(z)\frac{q(z)}{|q(z)|}\right|^2}{1-|\mu_f(z)|^2}D_{g^{-1}}(f(z))dxdy,
    $$
    where $D_{g^{-1}}(w)$ is the dilatation of $g^{-1}$ at $w$ and $\mu_f$ is the Beltrami coefficient of the quasiconformal mapping $f$.
\end{theorem}

Following the inspiration from \cite{JQ2012}, we can obtain an upper estimate of the limiting Teichm\"uller distance between $\mathcal{R}_{q,X}(t)$ and $\mathcal{R}_{q^\prime,Y}(t)$.

\begin{lemma}\label{UpperEstimate}
    Let $\mathcal{R}_{q,X}(t)$ and $\mathcal{R}_{q^\prime,Y}(t)$ be two Teichm\"uller rays. The vertical measured foliations $V(q)=\sum_{j=1}^{N}a_jG_j$ and $V(q^\prime)=\sum_{j=1}^{N}b_jG_j$ are absolutely continuous, where each $G_j$ is a simple closed curve or a uniquely ergodic measure. Then 
    $$
    \limsup_{t\to\infty}d_{\mathcal{T}}(X_t,Y_t)\leq\max\left\{\frac{1}{2}\log\max_{1\leq j\leq N}\left\{\frac{m^\prime_j}{m_j},\frac{m_j}{m^\prime_j}\right\},d_{\overline{\mathcal{T}}}(X_\infty,Y_\infty)\right\}.
    $$
\end{lemma}

\begin{proof}
    Let $X_\infty=\bigcup_{i=1}^{n}X_{\infty,i}$ be the limit surface of the Teichm\"uller ray $\mathcal{R}_{q,X}(t)$ and $Y_\infty=\bigcup_{i=1}^{n}Y_{\infty,i}$ be the limit surface of $\mathcal{R}_{q^\prime,Y}(t)$. For a decomposition $X_t=\bigcup_{i=1}^{n}X_{t,i}$ of $X_t$, since $X_{t,i}$ can be isometrically embedded in $X_{\infty,i}$ while preserving the graph $\Gamma_{q,i}$, we treat $X_{t,i}$ as a subsurface of $X_{\infty,i}$. 
    
    For each surface $X_t$ along $\mathcal{R}_{q,X}(t)$, we choose an appropriate horizontal segment $\tau_t$ for each minimal component of $X_t\setminus\Gamma_q$ such that as described in \S\ref{TheDecomposition}, we can obtain a sequence of surfaces along $\mathcal{R}_{q,X}(t)$, still denote by $X_{t,i}$, which forms an exhaustion of $X_{\infty,i}$. Similarly, for the surface $Y_t$ along $\mathcal{R}_{q^\prime,Y}(t)$, we select the corresponding horizontal segment $\tau^\prime_t$ for each minimal component of $Y_t\setminus\Gamma_{q^\prime}$ such that the first return mappings on $\tau_t$ and $\tau^\prime_t$ are identical. By properly choosing the sequence $X_{t,i}$ of surfaces along $\mathcal{R}_{q,X}(t)$, we can ensure that the corresponding sequence $Y_{t,i}$ along $\mathcal{R}_{q^\prime,Y}(t)$ also forms an exhaustion of $Y_{\infty,i}$.
    
    By Lemma \ref{TheGluingMapping}, for sufficiently large $t$, there is a quasiconformal mapping $g_{t,i}:X_{t,i}\setminus\overline{\Gamma}_{q,i}\to Y_{t,i}\setminus\overline{\Gamma}_{q^\prime,i}$ with the dilatation 
    $$
    K(g_{t,i})\leq\max_{1\leq j\leq N}\left\{\frac{m_j^\prime}{m_j},\frac{m_j}{m_j^\prime}\right\}+O(\varepsilon)\text{ as } \varepsilon\to 0.
    $$ 
    Let $f_{\infty,i}:X_{\infty,i}\to Y_{\infty,i}$ be the Teichm\"uller mapping between $X_{\infty,i}$ and $Y_{\infty,i}$. Then $d_{\mathcal{T}_i}(X_{\infty,i},Y_{\infty,i})=\frac{1}{2}\log K(f_{\infty,i})$.

    Let $p$ be a puncture of $X_{\infty,i}$ enclosed by a boundary of $X_{t,i}$ and $p^\prime$ be the corresponding puncture on $Y_{\infty,i}$. We choose a neighborhood $U$ of $p$ and a holomorphic mapping $\phi$ such that $\phi(U)=\mathbb{D}^\ast=\{z\in\mathbb{C}\mid 0<z<1\}$ and $p$ is mapped to $0\in\mathbb{C}$. Let $U^\prime$ be a neighborhood of $p^\prime$ with $f_{\infty,i}(U)\subset U^\prime$ and $\psi$ be a holomorphic mapping such that $\psi(U^\prime)=\mathbb{D}^\ast$ and $p^\prime$ is mapped to $0\in\mathbb{C}$. Each connected component of $X_{\infty,i}\setminus X_{t,i}$ is a region containing a puncture of $X_{\infty,i}$. Denote by $U_t\subset X_{\infty,i}\setminus X_{t,i}$ the region containing $p$, such that $\phi(U_t)\subset\mathbb{D}^\ast$ for sufficiently large $t$. Let $U_t^\prime\subset Y_{\infty,i}\setminus Y_{t,i}$ be the corresponding region containing $p^\prime$.

    Let $C_r=\{z\in\mathbb{C}\mid |z|=r\}$, $D_r=\{z\in\mathbb{C}\mid |z|<r\}$ and $A_{r,r^\prime}=\{z\in\mathbb{C}\mid r\leq|z|<r^\prime\}$. For any sufficiently large $t$, We choose two circles $C_{r_t}$ and $C_{r_2}$ with $r_2>r_t$ such that $\phi(U_t)\subset D_{r_t}$ and $\psi\circ g_{t,i}\circ\phi^{-1}(C_{r_t})\subset\psi\circ f_{\infty,i}\circ\phi^{-1}(D_{r_2})$. There is a conformal mapping $\psi^\prime$ such that the $\psi\circ f_{\infty,i}\circ\phi^{-1}(D_{r_2})\setminus\psi\circ g_{t,i}\circ\phi^{-1}(D_{r_t})$ is mapped onto an annulus $A_{r_3,r_4}$. Since $g_{t,i}$ and $f_{\infty,i}$ are quasiconformal mappings, we can show that $\psi^\prime\circ\psi\circ g_{t,i}\circ\phi^{-1}$ is a quasisymmetric map from $C_{r_t}$ to $C_{r_3}$ and $\psi^\prime\circ\psi\circ f_{\infty,i}\circ\phi^{-1}$ is a quasisymmetric map from $C_{r_2}$ to $C_{r_4}$. By Lemma $1$ in \cite{JQ2012}, there exists a quasiconformal mapping $\Phi^\prime_t$ from the annulus $A_{r_t,r_2}$ to $A_{r_3,r_4}$ such that $\Phi^\prime_t$ is equal to $\psi^\prime\circ\psi\circ g_{t,i}\circ\phi^{-1}$ on $C_{r_t}$ and equal to $\psi^\prime\circ\psi\circ f_{\infty,i}\circ\phi^{-1}$ on $C_{r_2}$. Let 
    \begin{equation*}
        h_t=
        \begin{cases}
            \psi\circ g_{t,i}\circ\phi^{-1}(z), & z\in C_{r_t}, \\
            \psi\circ f_{\infty,i}\circ\phi^{-1}(z), & z\in C_{r_2}.
        \end{cases}
    \end{equation*}
    Then $\Phi_t=\psi^{\prime-1}\circ\Phi^\prime_t$ is a quasiconformal extension of $h_t$ to the annulus $A_{r_t,r_2}$. Let $\Psi_t\in[\Phi_t]$ be an extremal quasiconformal extension of $h_t$ to the annulus $A_{r_t,r_2}$ and let 
    \begin{equation*}
        F_t(z)=
        \begin{cases}
            \psi\circ g_{t,i}\circ\phi^{-1}(z), & z\in D_{r_t}\setminus\phi(U_t); \\
            \Psi_t(z), & z\in A_{r_t,r_2}; \\
            \psi\circ f_{\infty,i}\circ\phi^{-1}(z), & z\in D_1\setminus D_{r_2}. 
        \end{cases}
    \end{equation*}
    For any $\varepsilon>0$, we show that 
    $$
    K(\Psi_t)\leq\max\{K(\psi\circ g_{t,i}\circ\phi^{-1}),K(\psi\circ f_{\infty,i}\circ\phi^{-1})\}+\varepsilon=\max\{K(g_{t,i}),K(f_{\infty,i})\}+\varepsilon
    $$
    for sufficiently small $r_t>0$ as $t\to\infty$.
    
    By contradiction, suppose that for all $r_t>0$ as $t\to\infty$, 
    \begin{equation}\label{assumption}
        K(\Psi_t)>\max\{K(g_{t,i}),K(f_{\infty,i})\}+\varepsilon.
    \end{equation}
    We can pick $r_t^\prime$ and $r_2^\prime$ with $r_t<r_t^\prime<r_2^\prime<r_2$ such that there similarly exists an extremal quasiconformal extension $\Phi_{r_t^\prime,r_2^\prime}$ from $A_{r_t^\prime,r_2^\prime}$ to $\psi\circ f_{\infty,i}\circ\phi^{-1}(D_{r_2^\prime})\setminus\psi\circ g_{t,i}\circ\phi^{-1}(D_{r_t^\prime})$ with $\Phi_{r_t^\prime,r_2^\prime}=\psi\circ g_{t,i}\circ\phi^{-1}$ on $C_{r_t^\prime}$ and $\Phi_{r_t^\prime,r_2^\prime}=\psi\circ f_{\infty,i}\circ\phi^{-1}$ on $C_{r_2^\prime}$. Let
    \begin{equation*}
        G_{t}(z)=
        \begin{cases}
            \psi\circ g_{t,i}\circ\phi^{-1}(z), & z\in A_{r_t,r_t^\prime}; \\
            \Phi_{r_t^\prime,r_2^\prime}(z), & z\in A_{r_t^\prime,r_2^\prime}; \\
            \psi\circ f_{\infty,i}\circ\phi^{-1}(z), & z\in A_{r_2^\prime,r_2}. 
        \end{cases}
    \end{equation*}
    Then the boundary dilatation of $h_t$ is 
    \begin{equation}\label{H(h_t)}
        H(h_t)\leq H^\ast(G_t)\leq\max\{K(g_{t,i}),K(f_{\infty,i})\}.
    \end{equation}
    Therefore, $K(\Psi_t)>H(h_t)$. By Theorem \ref{FrameMappingThm}, $\Psi_t(z)$ is an extremal quasiconformal mapping with Beltrami coefficient $\mu_{r_t}=k_{r_t}\frac{\bar{q}_{r_t}}{|q_{r_t}|}$ $(0<k_{r_t}<1)$, where $q_{r_t}=q_{r_t}(z)dz^2$ is the associated holomorphic quadratic differential with $\|q_{r_t}\|=1$.

    For each sufficiently large $t$, we can choose the $r_t>0$ such that $r_t\to 0$ as $t\to\infty$. We show that the sequence $q_{r_t}$ converges to $0$ uniformly on any compact subset of $D_{r_2}\setminus\{0\}$ as $r_t\to 0$. 

    By contradiction, suppose that there exist a sequence $\{r_{t,n}\}$ decreasing to $0$ and a non-zero holomorphic mapping $q_0$ on $D_{r_2}\setminus\{0\}$ such that $q_{r_{t,n}}\to q_0$ as $n\to\infty$, where $q_{r_{t,n}}$ is the associated holomorphic quadratic differential of the extremal quasiconformal mapping $\Psi_{t,n}$. Since $\{K(\Psi_{t,n})\}$ is non-increasing and bounded, then $k_{r_{t,n}}\to k_0$ and the Beltrami coefficient $\mu_{t,n}$ of $\Psi_{t,n}$ converges to $\mu_0=k_0\frac{\bar{q}_0}{|q_0|}$ uniformly on any compact subset of $D_{r_2}\setminus\{0\}$ as $n\to\infty$. 

    Since these mappings $F_{t,n}$ and their dilatations are uniformly bounded, for any compact subset $E_{r_2}$ of $D_{r_2}\setminus\{0\}$, there is a subsequence of $F_{t,n}$ with $E_{r_2}\subset A_{r_{t,n},r_2}$ such that the subsequence of $F_{t,n}$ is a normal family on $E_{r_2}$. Using Cantor diagonalization process, we can get a subsequence of $F_{t,n}$ which converges to a quasiconformal mapping $F_0$ uniformly on any compact subset of $D_{r_2}\setminus\{0\}$. Then $F_0$ is a quasiconformal mapping with Beltrami coefficient $\mu_0=k_0\frac{\bar{q}_0}{|q_0|}$. Since $\|q_0\|\leq\lim_{n\to\infty}\|q_{t,n}\|=1$, $F_0$ is an extremal quasiconformal mapping. By the assumption \eqref{assumption}, we obtain that 
    \begin{equation}\label{K(F_0)}
        K(F_0)\geq K(f_{\infty,i})+\varepsilon.
    \end{equation}
    From the construction of $F_{t,n}$, we get 
    $$
    F_0\in\left[\psi\circ f_{\infty,i}\circ\phi^{-1}\big|_{D_{r_2}\setminus\{0\}}\right].
    $$
    Then \eqref{K(F_0)} contradicts that $F_0$ is an extremal quasiconformal mapping. Therefore, the sequence $q_{r_t}$ converges to $0$ uniformly on any compact subset of $D_{r_2}\setminus\{0\}$ as $r_t\to 0$.

    It follows from \eqref{H(h_t)} that there is a compact subset $E$ of the annulus $A_{r_t,r_2}$ such that
    \begin{equation}\label{K(G_t)}
        K\left(G_t\big|_{A_{r_t,r_2}\setminus E}\right)<\max\{K(g_{t,i}),K(f_{\infty,i})\}+\frac{\varepsilon}{2}.
    \end{equation}
    Since $\Psi_t$ is homotopic to $G_t$ on $A_{r_t,r_2}$ modulo the boundary, applying Theorem \ref{MainInequality} to $\Psi_t$ and $G_t$ on $A_{r_t,r_2}$, we obtain
    \begin{align*}
        1=\|q_{r_t}\| & \leq\iint_{A_{r_t,r_2}}|q_{r_t}(z)|\frac{\left|1-\mu_{r_t}(z)\frac{q_{r_t}(z)}{|q_{r_t}(z)|}\right|^2}{1-|\mu_{r_t}(z)|^2}D_{G_t^{-1}}(\Psi_t(z))dxdy \\
        & =\iint_{A_{r_t,r_2}}\frac{|q_{r_t}(z)|}{K(\Psi_t)}D_{G_t^{-1}}(\Psi_t(z))dxdy.
    \end{align*}
    Thus,
    \begin{equation}\label{K(Psi_t)}
        \begin{aligned}
            K(\Psi_t)\leq &\iint_{A_{r_t,r_2}}|q_{r_t}(z)|D_{G_t^{-1}}(\Psi_t(z))dxdy \\
            = & \iint_{\Psi_t^{-1}\circ G_t(E)}|q_{r_t}(z)|D_{G_t^{-1}}(\Psi_t(z))dxdy \\
            & +\iint_{A_{r_t,r_2}\setminus\Psi_t^{-1}\circ G_t(E)}|q_{r_t}(z)|D_{G_t^{-1}}(\Psi_t(z))dxdy.
        \end{aligned}
    \end{equation}
    From the definitions of $\Psi_t$ and $G_t$, the dilatation $K(\Psi_t^{-1}\circ G_t)$ is uniformly bounded for any $r_t$. Thus, $\Psi_t^{-1}\circ G_t(E)$ is contained in a compact subset of $D_{r_2}\setminus\{0\}$ for any $r_t$. Since $q_{r_t}$ degenerates to $0$ as $r_t\to 0$, 
    \begin{equation}\label{Psi-1G(E)}
        \iint_{\Psi_t^{-1}\circ G_t(E)}|q_{r_t}(z)|D_{G_t^{-1}}(\Psi_t(z))dxdy\leq\frac{\varepsilon}{2}
    \end{equation}
    for all sufficiently small $r_t$. By the definition of $G_t$ and \eqref{K(G_t)},
    \begin{equation}\label{A-Psi-1G(E)}
        \iint_{A_{r_t,r_2}\setminus\Psi_t^{-1}\circ G_t(E)}|q_{r_t}(z)|D_{G_t^{-1}}(\Psi_t(z))dxdy\leq\max\{K(g_{t,i}),K(f_{\infty,i})\}+\frac{\varepsilon}{2}.
    \end{equation}
    Therefore, by \eqref{K(Psi_t)}, \eqref{Psi-1G(E)} and \eqref{A-Psi-1G(E)}, we obtain that for all sufficiently small $r_t$, 
    $$
    K(\Psi_t)\leq\max\{K(g_{t,i}),K(f_{\infty,i})\}+\varepsilon.
    $$
    This contradicts to the assumption \eqref{assumption}. Thus, for any $\varepsilon>0$, we have
    $$
    K(F_t)\leq\max\{K(g_{t,i}),K(f_{\infty,i})\}+\varepsilon
    $$
    for sufficiently small $r_t$ as $t\to\infty$. 

    If we consider the mapping $\psi^{-1}\circ F_t\circ\phi$ on the neighborhood of the puncture $p$ of $X_{\infty,i}$, the mapping $g_{t,i}$ can be glued to the mapping $f_{\infty,i}$ by $\psi^{-1}\circ F_t\circ\phi$ for sufficiently large $t$. The same case holds for any puncture of each $X_{\infty,i}$ that is formed by a semi-infinite cylinder or some half planes. Then we obtain a mapping $g_{t,i}^\prime:X_{t,i}\to Y_{t,i}$ with $g_{t,i}^\prime\big|_{\partial X_{t,i}}=g_{t,i}\big|_{\partial X_{t,i}}$ on the boundaries of $X_{t,i}$, and the dilatation of $g_{t,i}^\prime$ satisfies
    $$
    K(g_{t,i}^\prime)\leq\max\left\{\max_{1\leq j\leq N}\left\{\frac{m_j^\prime}{m_j},\frac{m_j}{m_j^\prime}\right\},K(f_{\infty,i})\right\}+O(\varepsilon) \text{ as } \varepsilon\to 0.
    $$
    By the construction of $g_{t,i}$ in Lemma \ref{TheGluingMapping}, the mappings $g_{t,i}^\prime$ agree along the boundaries of $X_{t,i}$. We glue the mappings $g_{t,i}^\prime$ along the boundaries of $X_{t,i}$ compositing with some Dehn-twists if necessary. Then for sufficiently large $t$, we get a mapping $g_t^\prime:X_t\to Y_t$ homotopic to $f_{t,2}\circ f_{t,1}^{-1}$, where $f_{t,1}$ and $f_{t,2}$ are the markings of $X_t$ and $Y_t$ respectively. The dilatation of $g_t^\prime$ is
    $$
    K(g_t^\prime)\leq \max\left\{\max_{1\leq j\leq N}\left\{\frac{m_j^\prime}{m_j},\frac{m_j}{m_j^\prime}\right\},\max_{1\leq j\leq N}K(f_{\infty,i})\right\}+O(\varepsilon) \text{ as } \varepsilon\to 0.
    $$
    This implies that 
    $$
    \limsup_{t\to\infty}d_{\mathcal{T}}(X_t,Y_t)\leq\max\left\{\frac{1}{2}\log\max_{1\leq j\leq N}\left\{\frac{m^\prime_j}{m_j},\frac{m_j}{m^\prime_j}\right\},d_{\overline{\mathcal{T}}}(X_\infty,Y_\infty)\right\}.
    $$
\end{proof}

\section{Lower estimate of the limiting Teichm\"uller distance}

We give a lower estimate of the limiting Teichm\"uller distance for a pair of Teichm\"uller rays $\mathcal{R}_{q,X}(t)$ and $\mathcal{R}_{q^\prime,Y}(t)$, where the vertical measured foliations of $q$ and $q^\prime$ are absolutely continuous. The vertical measured foliations $V(q)$ and $V(q^\prime)$ can be written as 
$$
V(q)=\sum_{j=1}^{N}a_jG_j, \text{ and } V(q^\prime)=\sum_{j_1}^{N}b_jG_j,
$$
where $G_j$ is a simple closed curve or an ergodic measure, and $a_j$, $b_j$ are positive real numbers.

Recall the definition of the extremal length of a family $\Gamma$ of rectifiable curves in a domain $D$ of a Riemann surface. Let $\rho=\rho(z)|dz|$ be a Borel measurable conformal metric on $D$. Then the length of a rectifiable curve $\gamma\in\Gamma$ is 
$$
\ell_\rho(\gamma)=\int_{\gamma}\rho(z)|dz|,
$$
and the area of $D$ is 
$$
\mathrm{Area}_{\rho}(D)=\iint_{D}\rho(z)^2dxdy.
$$
The extremal length of $\Gamma$ in $D$ is defined by 
$$
\lambda_D(\Gamma)=\sup_{\rho}\frac{\inf_{\gamma\in\Gamma}\ell_{\rho}(\gamma)^2}{\mathrm{Area}_{\rho}(D)},
$$
where $\rho$ takes over all Borel measurable conformal metric on $D$ with $\mathrm{Area}_{\rho}(D)<\infty$.

The extremal length of $\Gamma$ is independent of the domain containing the $\Gamma$ by the definition of extremal length. For two families of curves $\Gamma$ and $\Gamma^\prime$ in $D$, if each $\gamma\in\Gamma$ contains a $\gamma^\prime\in\Gamma^\prime$, then $\lambda_D(\Gamma)\geq\lambda_D(\Gamma^\prime)$. The extremal length has the quasiconformal distortion property which is
$$
\frac{1}{K}\lambda_D(\Gamma)\leq\lambda_{D^\prime}(f(\Gamma))\leq K\lambda_D(\Gamma),
$$
where $f$ is a $K$-quasiconformal mapping from $D$ to $D^\prime$.

Let $\alpha\in\mathcal{S}$ be a simple closed curve on $S$ and $X=[X,f]\in\mathcal{T}(S)$. The extremal length $\mathrm{Ext}_X(\alpha)$ of $\alpha$ on $X$ is defined as
$$
\mathrm{Ext}_X(\alpha)=\sup_{\rho}\frac{\ell_\rho(\alpha)^2}{\mathrm{Area}_\rho(X)},
$$
where
$$
\ell_\rho(\alpha)=\inf_{\alpha^\prime\sim f(\alpha)}\int_{\alpha^\prime}\rho(z)|dz|, \quad \mathrm{Area}_{\rho}(X)=\iint_{X}\rho(z)^2dxdy,
$$
and $\rho$ ranges over all Borel measurable conformal metric on $X$ with $\mathrm{Area}_{\rho}(X)<\infty$. There is another “geometric” definition as follows.
$$
\mathrm{Ext}_X(\alpha):=\inf_{C_{\alpha}}\frac{1}{\mathrm{Mod}(C_{\alpha})},
$$
where $C_{\alpha}$ ranges over all embedded cylinders on $X$ whose core curve is isotopic to $f(\alpha)$, and $\mathrm{Mod}(C_{\alpha})$ is the modulus of the cylinder $C_{\alpha}$ defined by the ratio of the height and circumference of $C_{\alpha}$.

The extremal length $\mathrm{Ext}_X(t\alpha)$ of a weighted simple closed curve $t\alpha\in\mathbb{R}_{\geq 0}\otimes\mathcal{S}$ is defined by 
\begin{equation}\label{ExtremalLength}
    \mathrm{Ext}_X(t\alpha)=t^2\mathrm{Ext}_X(\alpha).
\end{equation}
Kerckhoff \cite{Ker1980} showed that the extremal length function of $t\alpha\in\mathbb{R}_{\geq 0}\otimes\mathcal{S}$, defined as \eqref{ExtremalLength}, can extend continuously to $\mathcal{MF}(S)$ satisfying 
$$
\mathrm{Ext}_X(t\mathcal{F})=t^2\mathrm{Ext}_X(\mathcal{F}),
$$
for any $\mathcal{F}\in\mathcal{MF}(S)$ and $X\in\mathcal{T}(S)$. Kerckhoff also gave a useful formula of the Teichm\'uller distance by extremal length as follows.

\begin{theorem}[\cite{Ker1980}]\label{Kerckhoff'sFormula}
    Let $X,Y\in\mathcal{T}(S)$ be two Riemann surfaces. The Teichm\"uller distance between $X$ and $Y$ is
    $$
    d_{\mathcal{T}}(X,Y)=\frac{1}{2}\log\sup_{\mathcal{F}\in\mathcal{MF}(S)\setminus\{0\}}\frac{\mathrm{Ext}_Y(\mathcal{F})}{\mathrm{Ext}_X(\mathcal{F})}.
    $$
\end{theorem}

Let $X=[X,f]\in\mathcal{T}(S)$ and $\mathcal{G}\in\mathcal{MF}(X)$. For any $\mathcal{F}\in\mathcal{MF}(S)$, we define the intersection number $i(\mathcal{G},\mathcal{F})$ on $X$ as 
$$
i(\mathcal{G},\mathcal{F})=i(\mathcal{G},f_\ast(\mathcal{F})).
$$
Let $\mathcal{R}_{q,X}(t)$ be a Teichm\"uller ray and $V(q)=\sum_{j=1}^{N}a_jG_j$ be the vertical measured foliation where $a_j\geq 0$. We set
$$
\mathcal{E}_{q,X}(\mathcal{F})=\left\{\sum_{j=1}^{N}\frac{a_ji(G_j,\mathcal{F})^2}{i(G_j,H(q))}\right\}^{\frac{1}{2}}
$$
for any $\mathcal{F}\in\mathcal{MF}(S)$.

\begin{theorem}[\cite{Wal2019}]\label{limitofExt}
    Let $\mathcal{R}_{q,X}(t)$ be a Teichm\"uller ray and $V(q)=\sum_{j=1}^{N}a_jG_j$ be the vertical measured foliation where $a_j\geq 0$. Then for any $\mathcal{F}\in\mathcal{MF}(S)$, there is 
    $$
    \lim_{t\to\infty}e^{-2t}\mathrm{Ext}_{X_t}(\mathcal{F})=\sum_{j=1}^{N}\frac{a_ji(G_j,\mathcal{F})^2}{i(G_j,H(q))}=\mathcal{E}_{q,X}(\mathcal{F})^2.
    $$
\end{theorem}

\begin{lemma}[\cite{Wal2019}]\label{ratioofE}
    Let $\mathcal{R}_{q,X}(t)$, $\mathcal{R}_{q^\prime,Y}(t)$ be two Teichm\"uller rays and $V(q)=\sum_{j=1}^{N}a_jG_j$ be the vertical measured foliation where $a_j>0$. If the vertical measured foliation $V(q^\prime)$ can be written as $V(q^\prime)=\sum_{j=1}^{N}b_jG_j$ where $b_j\geq 0$, then 
    $$
    \sup_{\mathcal{F}\in\mathcal{MF}(S)\setminus Z}\frac{\mathcal{E}_{q^\prime,Y}(\mathcal{F})^2}{\mathcal{E}_{q,X}(\mathcal{F})^2}=\max_{1\leq j\leq N}\frac{b_ji(G_j,H(q))}{a_ji(G_j,H(q^\prime))},
    $$
    where $Z=\{\mathcal{F}\in\mathcal{MF}(S)\mid \mathcal{E}_{q,X}(\mathcal{F})=\mathcal{E}_{q^\prime,Y}(\mathcal{F})=0\}$. Otherwise, the supremum is $+\infty$.
\end{lemma}

\begin{xrem}
    We note that if $V(q)$ and $V(q^\prime)$ are absolutely continuous, then 
    $$
    \sup_{\mathcal{F}\in\mathcal{MF}(S)\setminus Z}\frac{\mathcal{E}_{q^\prime,Y}(\mathcal{F})^2}{\mathcal{E}_{q,X}(\mathcal{F})^2} \text{ and } \sup_{\mathcal{F}\in\mathcal{MF}(S)\setminus Z}\frac{\mathcal{E}_{q,X}(\mathcal{F})^2}{\mathcal{E}_{q^\prime,Y}(\mathcal{F})^2}
    $$
    are both bounded.
\end{xrem}

The following estimate is gave by Amano in \cite{Ama2014}. For the completeness of the paper, we state the result and give the proof.

\begin{lemma}\label{LowerEstimate-Modulus}
    Let $\mathcal{R}_{q,X}(t)$ and $\mathcal{R}_{q^\prime,Y}(t)$ be two Teichm\"uller rays. If the vertical measured foliations $V(q)=\sum_{j=1}^{N}a_jG_j$ and $V(q^\prime)=\sum_{j=1}^{N}b_jG_j$ are absolutely continuous. Then
    $$
    \liminf_{t\to\infty}d_{\mathcal{T}}(X_t,Y_t)\geq\frac{1}{2}\log\max_{1\leq j\leq N}\left\{\frac{m^\prime_j}{m_j},\frac{m_j}{m^\prime_j}\right\}.
    $$
\end{lemma}

\begin{proof}
    By Theorem \ref{Kerckhoff'sFormula}, \ref{limitofExt} and Lemma \ref{ratioofE}, we obtain that
    \begin{align*}
        \liminf_{t\to\infty}d_{\mathcal{T}}(X_t,Y_t) & =\liminf_{t\to\infty}\frac{1}{2}\log\sup_{\mathcal{F}\in\mathcal{MF}(S)\setminus\{0\}}\frac{\mathrm{Ext}_{Y_t}(\mathcal{F})}{\mathrm{Ext}_{X_t}(\mathcal{F})} \\
        & \geq \frac{1}{2}\log\sup_{\mathcal{F}\in\mathcal{MF}(S)\setminus Z}\liminf_{t\to\infty}\frac{e^{-2t}\mathrm{Ext}_{Y_t}(\mathcal{F})}{e^{-2t}\mathrm{Ext}_{X_t}(\mathcal{F})} \\
        & = \frac{1}{2}\log\max_{1\leq j\leq N}\frac{m^\prime_j}{m_j}.
    \end{align*}
    Since the symmetry of the distance, we can get the desired estimate.
\end{proof}

\begin{lemma}\label{LowerEstimate-LimitingSurfaces}
    Let $\mathcal{R}_{q,X}(t)$ and $\mathcal{R}_{q^\prime,Y}(t)$ be two Teichm\"uller rays, and $V(q)$ and $V(q^\prime)$ are absolutely continuous. Let $f_t:X_t\to Y_t$ be the Teichm\"uller mapping between $X_t$ and $Y_t$. Then there is a quasiconformal mapping $f_\infty:X_\infty\to Y_\infty$ induced by the sequence $f_t$, where $X_\infty$ and $Y_\infty$ are the limit surfaces of $\mathcal{R}_{q,X}(t)$ and $\mathcal{R}_{q^\prime,Y}(t)$ respectively. Moreover, 
    $$
    \liminf_{t\to\infty}d_{\mathcal{T}}(X_t,Y_t)\geq d_{\overline{\mathcal{T}}}(X_\infty,Y_\infty).
    $$
\end{lemma}

\begin{proof}
    Let $X_t=\bigcup_{i=1}^{n}X_{t,i}$ be the decomposition of $X_t$ as in \S\ref{TheDecomposition} and $Y_t=\bigcup_{i=1}^{n}Y_{t,i}$ be the corresponding decomposition of $Y_t$. The decomposition of the limit surfaces are $X_\infty=\bigcup_{i=1}^{n}X_{\infty,i}$ and $Y_\infty=\bigcup_{i=1}^{n}Y_{\infty,i}$. Under the singular flat metric induced by $e^{2t}q_t$, the subsurface $X_{t,i}\subset X_t$ can be isometrically embedded in $X_{\infty,i}$ while preserving the graph $\Gamma_{q,i}$. We treat $X_{t,i}$ as a subsurface of $X_{\infty,i}$.
    
    We select an appropriate horizontal segment $\tau_t$ for each minimal component of $X_t\setminus\Gamma_q$ and denote by  $\tau^\prime_t$ the corresponding horizontal segment on the corresponding minimal component of $Y_t\setminus\Gamma_{q^\prime}$. These segments are chosen such that the first return mappings on $\tau_t$ and $\tau^\prime_t$ coincide. Thus, we can obtain a sequence of surfaces along the Teichm\"uller ray $\mathcal{R}_{q,X}(t)$, still denote by $X_{t,i}$ for simplicity, which forms an exhaustion of the surface $X_{\infty,i}$. The corresponding sequence $Y_{t,i}$ also forms an exhaustion of the surface $Y_{\infty,i}$.

    Since the surfaces $X_{\infty,i}$ and $Y_{\infty,i}$ are of the same type, we have
    $$
    \chi(Y_{\infty,i})=\chi(X_{\infty,i})=2-2g_i-n_i,
    $$
    where $g_i$ is the genus of $X_{\infty,i}$ and $n_i$ is the number of punctures on $X_{\infty,i}$. If $g_i=0$ and $n_i\leq 3$, the Teichm\"uller space containing $X_{\infty,i}$ and $Y_{\infty,i}$ is trivial. Then the Teichm\"uller distance $d_{\mathcal{T}_i}(X_{\infty,i},Y_{\infty,i})=0$.

    We consider the case that the Teichm\"uller space containing $X_{\infty,i}$ and $Y_{\infty,i}$ is not trivial. Then there exist two non-trivial and non-peripheral simple closed curves $\alpha_i$ and $\beta_i$ on $X_{\infty,i}$ such that the intersection number $i(\alpha_i,\beta_i)\neq 0$. Recall that 
    $$
    \chi(\Gamma_{q,i})=\chi(X_{t,i})=\chi(X_{\infty,i})=2-2g_i-n_i.
    $$
    By the construction of $X_{\infty,i}$, the surface $X_{\infty,i}$ can shrink to the graph $\Gamma_{q,i}$. Then there exist two intersecting simple closed curves $\overline{\alpha}_i,\overline{\beta}_i\subset\Gamma_{q,i}$ consisting of saddle connections of $\Gamma_{q,i}$ such that $\overline{\alpha}_i$ and $\overline{\beta}_i$ are isotopic to $\alpha_i$ and $\beta_i$ on $X_{\infty,i}$, respectively. Since the vertical measured foliations $V(q)$ and $V(q^\prime)$ are topologically equivalent, there are two simple closed curves $\alpha^\prime_i$ and $\beta^\prime_i$ on $Y_{\infty,i}$ such that $\alpha^\prime_i$ and $\beta^\prime_i$ are homotopic to $f_t(\overline{\alpha}_i)$ and $f_t(\overline{\beta}_i)$, respectively. Similarly, there are two intersecting simple closed curves $\overline{\alpha}^\prime_i,\overline{\beta}^\prime_i\subset\Gamma_{q^\prime,i}$ consisting of saddle connections of $\Gamma_{q^\prime,i}$ such that $\overline{\alpha}^\prime_i$ and $\overline{\beta}^\prime_i$ are isotopic to $\alpha^\prime_i$ and $\beta^\prime_i$ on $Y_{\infty,i}$, respectively.

    Since $f_t(\overline{\alpha}_i)$ is homotopic to $\overline{\alpha}^\prime_i$ and $i(\overline{\alpha}^\prime_i,\overline{\beta}^\prime_i)\neq 0$, we obtain that the intersection $f_t(\overline{\alpha}_i)\cap\overline{\beta}^\prime_i$ is not empty. We pick a point $y_{t,i}\in f_t(\overline{\alpha}_i)\cap\overline{\beta}^\prime_i\subset\Gamma_{q^\prime,i}$, and let $x_{t,i}\in\Gamma_{q,i}$ be a point on $X_t$ such that $f_t(x_{t,i})=y_{t,i}$. Let $x_i\in\Gamma_{q,i}$ be a singularity of $X$ and $y_i\in\Gamma_{q^\prime,i}$ be a singularity of $Y$. Since the finite critical graphs $\Gamma_{q,i}$ and $\Gamma_{q^\prime,i}$ are preserved along the Teichm\"uller rays $\mathcal{R}_{q,X}(t)$ and $\mathcal{R}_{q^\prime,Y}(t)$, respectively, there is a constant $M>0$ such that  
    $$
    d_t(x_{t,i},x_i)<M \text{ and } d^\prime_t(y_{t,i},y_i)<M,
    $$
    where $d_t$ is the singular flat metric induced by $e^{2t}q_t$ on $X_t$ and $d^\prime_t$ is the singular flat metric induced by $e^{2t}q^\prime_t$ on $Y_t$. By Proposition \ref{TheGHconvergenceWithDifferentBasepoints}, the sequence $(X_t,x_{t,i})$ converges to $(X_{\infty,i},x_{\infty,i})$ in the sense of pointed Gromov-Hausdorff, where $x_{\infty,i}$ is a limit point of the sequence $x_{t,i}$ on $X_{\infty,i}$ as $t\to\infty$, and the sequence $(Y_t,y_{t,i})$ converges to $(Y_{\infty,i},y_{\infty,i})$ in the sense of pointed Gromov-Hausdorff, where $y_{\infty,i}$ is a limit point of the sequence $y_{t,i}$ on $Y_{\infty,i}$ as $t\to\infty$.

    Since $V(q)$ and $V(q^\prime)$ are absolutely continuous, the dilatation $K(f_t)$ of the Teichm\"uller mapping $f_t$ is bounded by the Theorem $3.2$ in \cite{Iva2001}. There is a subsequence of $f_t$, still denoted by $f_t$ for simplicity, satisfying 
    $$
    \lim_{t\to\infty}\frac{1}{2}\log K(f_t)=\liminf_{t\to\infty}d_{\mathcal{T}}(X_t,Y_t)=\frac{1}{2}\log K_{\infty}.
    $$
    Then for any $\varepsilon>0$, there is a $T>0$ such that for any $t>T$, $K(f_t)<K_{\infty}+\varepsilon$.

    We fix a $t_0>T$ and consider the normalized singular flat metric on each $X_t$ and $Y_t$. For any $t>t_0$, the subsurface $X_{t_0,i}\subset X_{t_0}$ can be isometrically embedded in $X_{t,i}\subset X_t$ while preserving the graph $\Gamma_{q,i}$. We show that there is a $T^\prime>t_0$ such that for any $t>T^\prime$, the image of $X_{t_0,i}$ under the Teichm\"uller mapping $f_t:X_t\to Y_t$ is contained in $Y_{t,i}\subset Y_t$, that is $f_t(X_{t_0,i})\subset Y_{t,i}\subset Y_t$. 

    We treat $X_{t_0,i}$ as a subsurface of $X_{t,i}\subset X_t$. By contradiction, assume that there always exists a sufficiently large $t>t_0$ such that $f_t(X_{t_0,i})$ is not contained in $Y_{t,i}$. Then, there is a point $\overline{x}\in\partial X_{t_0,i}$ such that $f_t(\overline{x})\notin Y_{t,i}$. Let $\beta$ be the geodesic arc connecting $x_{t,i}\in\Gamma_{q,i}$ and $\overline{x}\in\partial X_{t_0,i}$ under the normalized singular flat metric. Thus, the arc $f_t(\gamma)$ joining $y_{t,i}=f_t(x_{t,i})\in\Gamma_{q^\prime,i}$ and $f_t(\overline{x})$ intersects the boundary of $Y_{t,i}$. We can choose an annulus $A$ in $Y_{t,i}$ such that the boundary of $A$ is isotopic to the boundary of $Y_{t,i}$, and the arc $f_t(\gamma)$ traverses the annulus $A$. Let $Q$ be a quadrilateral in $X_{t_0,i}$ that contains the geodesic arc $\beta$, and $\beta$ connects a pair of opposite edges of $Q$. Let $\Gamma$ be the family of curves isotopic to $\beta$ in $Q$ that connect the pair of opposite edges of $Q$. We can choose the pair of opposite edges of $Q$ containing $x_{t,i}$ and $\overline{x}$ respectively to be sufficiently small such that the family of curves $f_t(\Gamma)$ traverses the annulus $A$. Then, by the quasiconformal distortion property of extremal length, we have
    \begin{equation}\label{lambda_{Y_t}(f_t(Gamma))}
        \frac{1}{K(f_t)}\lambda_{X_t}(\Gamma)\leq\lambda_{Y_t}(f_t(\Gamma))\leq K(f_t)\lambda_{X_t}(\Gamma).
    \end{equation}
    Since the extremal length of $\Gamma$ is independent of the domain containing the $\Gamma$, there is 
    \begin{equation}\label{lambda_{X_{t_0,i}}(Gamma)}
        \lambda_{X_t}(\Gamma)=\lambda_{X_{t_0,i}}(\Gamma).
    \end{equation}
    Let $\Gamma^\prime$ be the restriction of $f_t(\Gamma)$ in $A\subset Y_{t,i}$. Thus, each curve $\gamma\in f_t(\Gamma)$ contains a curve $\gamma^\prime\in\Gamma^\prime$. Then we have 
    $$
    \lambda_{Y_t}(f_t(\Gamma))\geq \lambda_{Y_t}(\Gamma^\prime).
    $$
    Let $\Gamma^{\prime\prime}$ be the family of curves connecting the two boundaries of $A$. Then there is $\Gamma^\prime\subset\Gamma^{\prime\prime}$, and 
    $$
    \lambda_{Y_t}(\Gamma^\prime)\geq\lambda_{Y_t}(\Gamma^{\prime\prime})=\lambda_A(\Gamma^{\prime\prime}).
    $$
    By the definition of extremal length, we have 
    $$
    \lambda_A(\Gamma^{\prime\prime})\geq\frac{\inf_{\gamma^{\prime\prime}\in\Gamma^{\prime\prime}}\ell_{\rho_t}(\gamma^{\prime\prime})^2}{\mathrm{Area}_{\rho_t}(A)}=\mathrm{Mod}(A),
    $$
    where $\rho_t$ is the singular flat metric induced by $e^{2t}q^\prime_t$ on $Y_t$, and $\mathrm{Mod}(A)$ is the modulus of the annulus $A$. The $A$ can be conformally mapped to an annulus $A_{r_1,r_2}=\{z\in\mathbb{C}\mid r_1\leq |z|\leq r_2\}$, and the modulus of $A$ is defined by $\frac{1}{2\pi}\log\frac{r_2}{r_1}$. Since the dilatation $K(f_t)$ is uniformly bounded, for sufficiently large $t>t_0$, we can choose the annulus $A$ in $Y_{t,i}$ such that 
    $$
    \mathrm{Mod}(A)>K(f_t)\lambda_{X_{t_0,i}}(\Gamma).
    $$
    This is a contradiction to \eqref{lambda_{Y_t}(f_t(Gamma))} and \eqref{lambda_{X_{t_0,i}}(Gamma)}. Then, there is a $T^\prime>t_0$ such that for any $t>T^\prime$, we have $f_t(X_{t_0,i})\subset Y_{t,i}\subset Y_t$.

    Since the surface $Y_{t,i}$ can be isometrically embedded into $Y_{\infty,i}$ under the normalized singular flat metrics, then the Teichm\"uller mapping $f_t:X_t\to Y_t$ induces a quasiconformal mapping from $X_{t_0,i}\subset X_{\infty,i}$ into $Y_{\infty,i}$ for any $t>T^\prime$. We still denote by $f_t$ the quasiconformal mapping from $X_{t_0,i}\subset X_{\infty,i}$ into $Y_{\infty,i}$ for simplicity. Assume that for any $t>T^\prime$, there exists a $t_1>T^\prime$ such that $f_{t_1}(X_{t_0,i})\nsubseteq Y_{t,i}$. By applying a similar argument as before, we arrive at a contradiction. Therefore, there exists a $t^\prime>T^\prime$ such that for all $t>t^\prime$, we have $f_t(X_{t_0,i})\subset Y_{t^\prime,i}\subset Y_{\infty,i}$. Then, there is a sequence of quasiconformal mappings $f_t$ from $X_{t_0,i}\subset X_{\infty,i}$ into $Y_{\infty,i}$, which forms a normal family. We can obtain a subsequence of $f_t$ that converges to a quasiconformal mapping from $X_{t_0,i}\subset X_{\infty,i}$ into $Y_{\infty,i}$.

    Let $C_{t_0}$ be a simple closed curve on the boundary of $X_{t_0,i}$ enclosing a puncture of $X_{\infty,i}$. The simple closed curve $C_{t_0}$ shrinks to the puncture of $X_{\infty,i}$ as $t_0\to\infty$. Then the simple closed curve $f_t(C_{t_0})$ also shrinks to a puncture of $Y_{\infty,i}$. Let $t_0$ tend to infinity. By Cantor diagonalization process, we can obtain a subsequence of $f_t$ which converges to a quasiconformal homeomorphism $f_{\infty,i}:X_{\infty,i}\to Y_{\infty,i}$ uniformly on any compact
    subset of $X_{\infty,i}$, and the dilatation $K(f_{\infty,i})<K_\infty+\varepsilon$. 

    Therefore, we can obtain a quasiconformal mapping $f_\infty:X_\infty\to Y_\infty$ induced by $f_t:X_t\to Y_t$. For each subsurface $X_{\infty,i}$ of $X_\infty$, if $g_i=0$ and $n_i\leq 3$, the restriction $f_\infty\big|_{X_{\infty,i}}$ is a conformal mapping from $X_{\infty,i}$ to $Y_{\infty,i}$. Otherwise, the restriction $f_\infty\big|_{X_{\infty,i}}=f_{\infty,i}$. Moreover, by the arbitrariness of $\varepsilon$ and the definition of Teichm\"uller distance between $X_\infty$ and $Y_\infty$, we can get that 
    $$
    \liminf_{t\to\infty}d_{\mathcal{T}}(X_t,Y_t)\geq d_{\overline{\mathcal{T}}}(X_\infty,Y_\infty).
    $$ 
\end{proof}

\begin{proof}[\textbf{Proof of Theorem} \ref{TheLimitDistance}]
    If the vertical measured foliations $V(q)$ and $V(q^\prime)$ are absolutely continuous, by Lemma \ref{LowerEstimate-Modulus} and Lemma \ref{LowerEstimate-LimitingSurfaces}, we obtain a lower estimate of the limiting Teichm\"uller distance, that is
    $$
    \liminf_{t\to\infty}d_{\mathcal{T}}(X_t,Y_t)\geq\max\left\{\frac{1}{2}\log\max_{1\leq j\leq N}\left\{\frac{m^\prime_j}{m_j},\frac{m_j}{m^\prime_j}\right\},d_{\overline{\mathcal{T}}}(X_\infty,Y_\infty)\right\}.
    $$
    Together with Lemma \ref{UpperEstimate}, we can get the desired equation.
    $$
    \lim_{t\to\infty}d_{\mathcal{T}}(X_t,Y_t)=\max\left\{\frac{1}{2}\log\max_{1\leq j\leq N}\left\{\frac{m^\prime_j}{m_j},\frac{m_j}{m^\prime_j}\right\},d_{\overline{\mathcal{T}}}(X_\infty,Y_\infty)\right\}.
    $$
    If $V(q)$ and $V(q^\prime)$ are not absolutely continuous, by the results in \cite{Iva2001} and \cite{LM2010}, the Teichm\"uller distance $d_{\mathcal{T}}(X_t,Y_t)$ tends to infinity as $t\to\infty$.
\end{proof}

\begin{proof}[\textbf{Proof of Corollary} \ref{TheAsymptoticCondition}]
    Under the assumption of Theorem \ref{TheLimitDistance}, If the two Teichm\"uller rays $\mathcal{R}_{q,X}(t)$ and $\mathcal{R}_{q^\prime,Y}(t)$ are asymptotic, we can assume that
    $$
    \lim_{t\to\infty}d_{\mathcal{T}}(X_t,Y_t)=0.
    $$
    By Theorem \ref{TheLimitDistance}, we get that $d_{\overline{\mathcal{T}}(X_\infty,Y_\infty)}=0$ and $m_j^\prime=m_j$ for any $j=1,\cdots,N$. Then the vertical measured foliations $V(q)$ and $V(q^\prime)$ are modularly equivalent and $X_\infty=Y_\infty$.

    Conversely, if $X_\infty=Y_\infty$ and $V(q)$ and $V(q^\prime)$ are modularly equivalent, the Teichm\"uller distance $d_{\overline{\mathcal{T}}(X_\infty,Y_\infty)}=0$, and there is constant $C>0$ such that $m_j^\prime=Cm_j$ for any $j=1,\cdots,N$. Then for $\sigma=-\frac{1}{2}\log C$,
    $$
    \lim_{t\to\infty}d_{\mathcal{T}}(X_t,Y_{t+\sigma})=\frac{1}{2}\log\max_{1\leq j\leq N}\left\{\frac{e^{2\sigma}m^\prime_j}{m_j},\frac{m_j}{e^{2\sigma}m^\prime_j}\right\}=0.
    $$
    This shows that the Teichm\"uller rays $\mathcal{R}_{q,X}(t)$ and $\mathcal{R}_{q^\prime,Y}(t)$ are asymptotic.
\end{proof}

\begin{proof}[\textbf{Proof of Corollary} \ref{Masur'sTheorem}]
    By the main theorem of \cite{HM1979}, there is a quadratic differential $q^\prime$ on $Y$ such that the vertical measured foliations $V(q^\prime)$ and $V(q)$ are modularly equivalent. Then there is a Teichm\"uller ray $\mathcal{R}_{q^\prime,Y}(t)$ starting from $Y$. Since the finite critical graph $\Gamma_q$ contains no simple closed curves, this implies that each component $X_{q,i}$ of the limit surface $X_\infty$ of $\mathcal{R}_{q,X}(t)$ is simply connected and has a puncture. The same case holds for the limit surface $Y_\infty$ of $\mathcal{R}_{q^\prime,Y}(t)$. Then we can obtain that 
    $$
    d_{\overline{\mathcal{T}}}(X_\infty,Y_\infty)=0.
    $$
    By Corollary \ref{TheAsymptoticCondition}, the Teichm\"uller rays $\mathcal{R}_{q,X}(t)$ and $\mathcal{R}_{q^\prime,Y}(t)$ are asymptotic.
\end{proof}

\section{Minimum value of the limiting Teichm\"uller distance}

The limit of the Teichm\"uller distance between two Teichm\"uller rays is related to the distance between two limit surfaces and the ratio of moudulus for the vertical measured foliations on the initial surfaces. In this section, we shift the initial points along the Teichm\"uller rays and obtain the minimum value of the limiting Teichm\"uller distance. The minimum value can be represented by the detour metric $\delta$ between the endpoints of the Teichm\"uller rays on the Gardiner-Masur boundary of $\mathcal{T}(S)$.

\subsection{The Gardiner-Masur boundary and the horofunction boundary}

We recall the Gardiner-Masur compactification of Teichm\"uller space $\mathcal{T}(S)$. Define the mapping 
\begin{equation*}
    \begin{array}{cccl}
        \varphi: & \mathcal{T}(S) & \to & \mathbb{R}_{\geq 0}^{\mathcal{S}}  \\
        & X & \mapsto & \left\{\mathrm{Ext}_X(\alpha)^{\frac{1}{2}}\right\}_{\alpha\in\mathcal{S}}. 
    \end{array}
\end{equation*}
Let $\pi:\mathbb{R}_{\geq 0}^{\mathcal{S}}\setminus\{0\}\to P\mathbb{R}_{\geq 0}^{\mathcal{S}}$ be the natural projection. Gardiner and Masur \cite{GM1991} showed that the composition $\varPhi=\pi\circ\varphi:\mathcal{T}(S)\to P\mathbb{R}_{\geq 0}^{\mathcal{S}}$ is an embedding and the closure $\overline{\varPhi(\mathcal{T}(S))}$ is compact. The closure $\overline{\varPhi(\mathcal{T}(S))}$ is called the Gardiner-Masur compactification of $\mathcal{T}(S)$ denoted by $\overline{\mathcal{T}(S)}^{GM}$ and the boundary of $\overline{\varPhi(\mathcal{T}(S))}$ is called the Gardiner-Masur boundary denoted by $\partial_{GM}\mathcal{T}(S)$.

The horofunction compactification of a metric space is introduced by Gromov in \cite{Gro1981}. We also refer to \cite{Wal2019} for more details. Let $(M,d)$ be a proper geodesic metric space which means that under the metric $d$, any closed ball is compact and each pair of points in $M$ is joined by a geodesic segment. Choose a basepoint $b\in M$, and for each point $z\in M$, we can define a function $\psi_z:M\to\mathbb{R}$ given by
$$
\psi_z(x):=d(x,z)-d(b,z), \text{ for any } x\in M.
$$
Let $C(M)$ be the space of continuous functions on $M$, which is endowed with the topology of uniform convergence on any compact subset of $M$. Then the mapping $\varPsi:M\to C(M)$ given by $\varPsi(z):=\psi_z$ is an embedding. The closure $\overline{\varPsi(M)}$ is compact in $C(M)$, which is called the horofunction compactification of $M$. The boundary of $\overline{\varPsi(M)}$ is called the horofunction boundary of $M$. We denote by $\partial_{hor}M$ the horofunction boundary of $M$, and call $\xi\in\partial_{hor}M$ a horofunction.

It is known that the Teichm\"uller space $\mathcal{T}(S)$ with $d_{\mathcal{T}}$ is a proper geodesic metric space. We can consider the horofunction compactification of $\mathcal{T}(S)$ which is denoted by $\overline{\mathcal{T}(S)}^{hor}$. The horofunction boundary of $\mathcal{T}(S)$ is denoted by $\partial_{hor}\mathcal{T}(S)$. Liu and Su \cite{LS2014} showed that the horofunction compactification of Teichm\"uller space with the Teichm\"uller metric is homeomorphic to the Gardiner-Masur compactification. This is also proved by Walsh in \cite{Wal2019}. Then we can treat the Gardiner-Masur compactification of Teichm\"uller space as the horofunction compactification.

\subsection{The detour metric}

We recall the detour metric $\delta$ which is defined on a subset of the horofunction boundary of the metric space $(M,d)$ that consists of the horofunctions called Busemann points. 

Let $\gamma:E\to M$ be a mapping into the metric space $(M,d)$, where $E$ is an unbounded subset of $\mathbb{R}_{\geq 0}$ containing $0$. The mapping $\beta$ is called an almost-geodesic ray on $M$ if for any $\varepsilon>0$, there exists a $T\geq 0$ such that 
$$
|d(\gamma(0),\gamma(s))+d(\gamma(s),\gamma(t))-t|<\varepsilon,
$$
for any $s,t\in E$ with $t\geq s\geq T$. Rieffel \cite{Rie2002} proved that every almost-geodesic ray of $(M,d)$ converges to a point in $\partial_{hor}M$. A horofunction which is the limit of an almost-geodesic ray is called a Busemann point in $\partial_{hor}M$. We denote by $\partial_{B}M$ the subset of $\partial_{hor}M$ consisting of all Busemann points.

For any two horofunctions $\xi,\eta\in\partial_{hor}M$, the detour cost is defined as 
$$
H(\xi,\eta):=\sup_{W\ni\xi}\inf_{x\in W}\left(d(b,x)+\eta(x)\right),
$$
where $W$ takes over all neighborhoods of $\xi$ in the horofunction compactification of $(M,d)$. There is an equivalent definition, that is
$$
H(\xi,\eta):=\inf_{\gamma}\liminf_{t\to\infty}\left(d(b,\gamma(t))+\eta(\gamma(t))\right),
$$
where the infimum is taken over all paths $\gamma:\mathbb{R}_{\geq 0}\to M$ converging to $\xi$. Walsh \cite{Wal2014} showed that the symmetrization of detour cost satisfies the axiom of the distance on $\partial_{B}M$. Then for any $\xi,\eta\in\partial_{B}M$, we can define the detour metric as
$$
\delta(\xi,\eta)=H(\xi,\eta)+H(\eta,\xi).
$$
The detour metric $\delta$ may take the value $+\infty$.

Let $\mathcal{R}_{q,X}(t)$ be a Teichm\"uller ray. Recall that for any $\mathcal{F}\in\mathcal{MF}(S)$, 
$$
\mathcal{E}_{q,X}(\mathcal{F})=\left\{\sum_{j=1}^{N}\frac{a_ji(G_j,\mathcal{F})^2}{i(G_j,H(q))}\right\}^{\frac{1}{2}}.
$$
By Theorem \ref{limitofExt}, the Teichm\"uller ray $\mathcal{R}_{q,X}(t)$ converges to the function $\hat{\mathcal{E}}_{q,X}=\pi\circ\mathcal{E}_{q,X}:\mathcal{MF}(S)\to P\mathbb{R}_{\geq 0}^{\mathcal{S}}$ in the Gardiner-Masur compactification of $\mathcal{T}(S)$ (see the Corollary $1$ in \cite{Wal2019}).

We denote by $\partial_B\mathcal{T}(S)$ the subset of $\partial_{hor}\mathcal{T}(S)$ which consists of Busemann points. It is obvious that any Teichm\"uller ray is an almost-geodesic ray in $\mathcal{T}(S)$. Since $\overline{\mathcal{T}(S)}^{hor}$ is homeomorphic to $\overline{\mathcal{T}(S)}^{GM}$, we regard the limit $\hat{\mathcal{E}}_{q,X}$ of $\mathcal{R}_{q,X}(t)$ in $\overline{\mathcal{T}(S)}^{GM}$ as the corresponding Busemann point in $\partial_B\mathcal{T}(S)$.

\begin{proposition}[\cite{Ama2014}]\label{DetourMetric} 
    Let $\mathcal{R}_{q,X}(t)$ and $\mathcal{R}_{q^\prime,Y}(t)$ be two Teichm\"uller rays. If $V(q)=\sum_{j=1}^{N}a_jG_j$ and $V(q^\prime)=\sum_{j=1}^{N}b_jG_j$ are absolutely continuous, then the detour metric between $\hat{\mathcal{E}}_{q,X}$ and $\hat{\mathcal{E}}_{q^\prime,Y}$ is represented by 
    $$
    \delta(\hat{\mathcal{E}}_{q,X},\hat{\mathcal{E}}_{q^\prime,Y})=\frac{1}{2}\log\max_{1\leq j\leq N}\frac{m_j^\prime}{m_j}+\frac{1}{2}\log\max_{1\leq j\leq N}\frac{m_j}{m_j^\prime}.
    $$
    If $V(q)$ and $V(q^\prime)$ are not absolutely continuous, then $\delta(\hat{\mathcal{E}}_{q,X},\hat{\mathcal{E}}_{q^\prime,Y})=+\infty$.
\end{proposition}

Then we give the proof of Proposition \ref{InfimumofLimitDistance}.

\begin{proof}[\textbf{Proof of Proposition} \ref{InfimumofLimitDistance}]
    Under the assumption of Theorem \ref{TheLimitDistance}, if the vertical measured foliations $V(q)$ and $V(q^\prime)$ are absolutely continuous, by Proposition \ref{DetourMetric}, we get that 
    \begin{align*}
        \frac{1}{2}\log\max_{1\leq j\leq N}\left\{\frac{m^\prime_j}{m_j},\frac{m_j}{m^\prime_j}\right\} & \geq \frac{1}{2}\log\left(\max_{1\leq j\leq N}\left(\frac{m^\prime_j}{m_j}\right)^{\frac{1}{2}}\cdot\max_{1\leq j\leq N}\left(\frac{m_j}{m^\prime_j}\right)^{\frac{1}{2}}\right) \\
        & =\frac{1}{2}\left(\frac{1}{2}\log\max_{1\leq j\leq N}\frac{m_j^\prime}{m_j}+\frac{1}{2}\log\max_{1\leq j\leq N}\frac{m_j}{m_j^\prime}\right) \\
        & =\frac{1}{2}\delta(\hat{\mathcal{E}}_{q,X},\hat{\mathcal{E}}_{q^\prime,Y}).
    \end{align*}
    The detour metric $\delta(\hat{\mathcal{E}}_{q,X},\hat{\mathcal{E}}_{q^\prime,Y})$ and $d_{\overline{\mathcal{T}}}(X_\infty,Y_\infty)$ are independent of the initial points of the Teichm\"uller rays. Therefore, by Theorem \ref{TheLimitDistance}, 
    $$
    \lim_{t\to\infty}d_{\mathcal{T}}(X_t,Y_{t+\sigma})\geq\max\left\{\frac{1}{2}\delta(\hat{\mathcal{E}}_{q,X},\hat{\mathcal{E}}_{q^\prime,Y}),d_{\overline{\mathcal{T}}}(X_\infty,Y_\infty)\right\},
    $$
    for any $\sigma\in\mathbb{R}$. We can choose the $\sigma$ as
    $$
    \sigma=\frac{1}{4}\log\frac{\max_{1\leq j\leq N}\frac{m_j}{m^\prime_j}}{\max_{1\leq j\leq N}\frac{m^\prime_j}{m_j}}.
    $$
    Thus, 
    \begin{align*}
        \max_{1\leq j\leq N}\frac{e^{2\sigma}m^\prime_j}{m_j} & =\max_{1\leq j\leq N}\left\{\frac{\left(\max_{1\leq j\leq N}\frac{m_j}{m^\prime_j}\right)^{\frac{1}{2}}\cdot m_j^\prime}{\left(\max_{1\leq j\leq N}\frac{m^\prime_j}{m_j}\right)^{\frac{1}{2}}\cdot m_j}\right\} \\
        & =\left(\max_{1\leq j\leq N}\frac{m_j}{m^\prime_j}\right)^{\frac{1}{2}}\cdot\left(\max_{1\leq j\leq N}\frac{m^\prime_j}{m_j}\right)^{\frac{1}{2}} \\
        & =\max_{1\leq j\leq N}\frac{m_j}{e^{2\sigma}m^\prime_j}.
    \end{align*}
    Therefore, we obtain that
    \begin{align*}
        \lim_{t\to\infty}d_{\mathcal{T}}(X_t,Y_{t+\sigma}) & =\max\left\{\frac{1}{2}\log\max_{1\leq j\leq N}\left\{\frac{e^{2\sigma}m^\prime_j}{m_j},\frac{m_j}{e^{2\sigma}m^\prime_j}\right\},d_{\overline{\mathcal{T}}}(X_\infty,Y_\infty)\right\} \\
        & =\max\left\{\frac{1}{2}\left(\frac{1}{2}\log\max_{1\leq j\leq N}\frac{m_j^\prime}{m_j}+\frac{1}{2}\log\max_{1\leq j\leq N}\frac{m_j}{m_j^\prime}\right),d_{\overline{\mathcal{T}}}(X_\infty,Y_\infty)\right\} \\
        & =\max\left\{\frac{1}{2}\delta(\hat{\mathcal{E}}_{q,X},\hat{\mathcal{E}}_{q^\prime,Y}),d_{\overline{\mathcal{T}}}(X_\infty,Y_\infty)\right\}.
    \end{align*}
\end{proof}

%\nocite{*}
\bibliography{Reference}
\bibliographystyle{abbrvnat-nourl}

\end{document}

%% file: figure/rectangles.pdf_tex
%% Creator: Inkscape 1.3 (0e150ed6c4, 2023-07-21), www.inkscape.org
%% Accompanies image file 'rectangles.pdf' (pdf, eps, ps)
\begingroup%
  \makeatletter%
  \providecommand\color[2][]{%
    \errmessage{(Inkscape) Color is used for the text in Inkscape, but the package 'color.sty' is not loaded}%
    \renewcommand\color[2][]{}%
  }%
  \providecommand\transparent[1]{%
    \errmessage{(Inkscape) Transparency is used (non-zero) for the text in Inkscape, but the package 'transparent.sty' is not loaded}%
    \renewcommand\transparent[1]{}%
  }%
  \providecommand\rotatebox[2]{#2}%
  \newcommand*\fsize{\dimexpr\f@size pt\relax}%
  \newcommand*\lineheight[1]{\fontsize{\fsize}{#1\fsize}\selectfont}%
  \ifx\svgwidth\undefined%
    \setlength{\unitlength}{472.38901213bp}%
    \ifx\svgscale\undefined%
      \relax%
    \else%
      \setlength{\unitlength}{\unitlength * \real{\svgscale}}%
    \fi%
  \else%
    \setlength{\unitlength}{\svgwidth}%
  \fi%
  \global\let\svgwidth\undefined%
  \global\let\svgscale\undefined%
  \makeatother%
  \begin{picture}(1,0.58570672)%
    \lineheight{1}%
    \setlength\tabcolsep{0pt}%
    \put(0,0){\includegraphics[width=\unitlength,page=1]{rectangles.pdf}}%
    \put(0.05459324,0.27665773){\color[rgb]{0,0,0}\makebox(0,0)[lt]{\lineheight{0}\smash{\begin{tabular}[t]{l}$R_1$\end{tabular}}}}%
    \put(0.10044806,0.27665774){\color[rgb]{0,0,0}\makebox(0,0)[lt]{\lineheight{0}\smash{\begin{tabular}[t]{l}$R_2$\end{tabular}}}}%
    \put(0.14116825,0.27665773){\color[rgb]{0,0,0}\makebox(0,0)[lt]{\lineheight{0}\smash{\begin{tabular}[t]{l}$R_3$\end{tabular}}}}%
    \put(0.18566366,0.27665774){\color[rgb]{0,0,0}\makebox(0,0)[lt]{\lineheight{0}\smash{\begin{tabular}[t]{l}$R_4$\end{tabular}}}}%
    \put(0.25505389,0.27665774){\color[rgb]{0,0,0}\makebox(0,0)[lt]{\lineheight{0}\smash{\begin{tabular}[t]{l}$R_5$\end{tabular}}}}%
    \put(0.31950403,0.27665774){\color[rgb]{0,0,0}\makebox(0,0)[lt]{\lineheight{0}\smash{\begin{tabular}[t]{l}$R_3$\end{tabular}}}}%
    \put(0.35704872,0.27665774){\color[rgb]{0,0,0}\makebox(0,0)[lt]{\lineheight{0}\smash{\begin{tabular}[t]{l}$R_2$\end{tabular}}}}%
    \put(0.40607902,0.27665774){\color[rgb]{0,0,0}\makebox(0,0)[lt]{\lineheight{0}\smash{\begin{tabular}[t]{l}$R_6$\end{tabular}}}}%
    \put(0.06718824,0.24463757){\color[rgb]{0,0,0}\makebox(0,0)[lt]{\lineheight{0}\smash{\begin{tabular}[t]{l}$R_5$\end{tabular}}}}%
    \put(0.13127004,0.24463759){\color[rgb]{0,0,0}\makebox(0,0)[lt]{\lineheight{0}\smash{\begin{tabular}[t]{l}$R_7$\end{tabular}}}}%
    \put(0.17202626,0.24463758){\color[rgb]{0,0,0}\makebox(0,0)[lt]{\lineheight{0}\smash{\begin{tabular}[t]{l}$R_8$\end{tabular}}}}%
    \put(0.21799142,0.24463757){\color[rgb]{0,0,0}\makebox(0,0)[lt]{\lineheight{0}\smash{\begin{tabular}[t]{l}$R_6$\end{tabular}}}}%
    \put(0.27352084,0.24463759){\color[rgb]{0,0,0}\makebox(0,0)[lt]{\lineheight{0}\smash{\begin{tabular}[t]{l}$R_1$\end{tabular}}}}%
    \put(0.32512374,0.24463759){\color[rgb]{0,0,0}\makebox(0,0)[lt]{\lineheight{0}\smash{\begin{tabular}[t]{l}$R_8$\end{tabular}}}}%
    \put(0.36561926,0.24463755){\color[rgb]{0,0,0}\makebox(0,0)[lt]{\lineheight{0}\smash{\begin{tabular}[t]{l}$R_7$\end{tabular}}}}%
    \put(0.40496434,0.24463759){\color[rgb]{0,0,0}\makebox(0,0)[lt]{\lineheight{0}\smash{\begin{tabular}[t]{l}$R_4$\end{tabular}}}}%
    \put(0.1261076,0.36970955){\color[rgb]{0,0,0}\makebox(0,0)[lt]{\lineheight{0}\smash{\begin{tabular}[t]{l}$c$\end{tabular}}}}%
    \put(0.3401384,0.36970954){\color[rgb]{0,0,0}\makebox(0,0)[lt]{\lineheight{0}\smash{\begin{tabular}[t]{l}$c$\end{tabular}}}}%
    \put(0.15659104,0.09437692){\color[rgb]{0,0,0}\makebox(0,0)[lt]{\lineheight{0}\smash{\begin{tabular}[t]{l}$d$\end{tabular}}}}%
    \put(0.35170343,0.12977524){\color[rgb]{0,0,0}\makebox(0,0)[lt]{\lineheight{0}\smash{\begin{tabular}[t]{l}$d$\end{tabular}}}}%
    \put(-0.00045279,0.2828844){\color[rgb]{0,0,0}\makebox(0,0)[lt]{\lineheight{0}\smash{\begin{tabular}[t]{l}$\tau_+$\end{tabular}}}}%
    \put(-0.00045277,0.22646087){\color[rgb]{0,0,0}\makebox(0,0)[lt]{\lineheight{0}\smash{\begin{tabular}[t]{l}$\tau_-$\end{tabular}}}}%
    \put(0.22463323,0.38507682){\color[rgb]{0,0,0}\makebox(0,0)[lt]{\lineheight{0}\smash{\begin{tabular}[t]{l}$e$\end{tabular}}}}%
    \put(0.0052557,0.25025645){\color[rgb]{0,0,0}\makebox(0,0)[lt]{\lineheight{0}\smash{\begin{tabular}[t]{l}$b$\end{tabular}}}}%
    \put(0.45334441,0.27604996){\color[rgb]{0,0,0}\makebox(0,0)[lt]{\lineheight{0}\smash{\begin{tabular}[t]{l}$b$\end{tabular}}}}%
    \put(0.26355389,0.14716485){\color[rgb]{0,0,0}\makebox(0,0)[lt]{\lineheight{0}\smash{\begin{tabular}[t]{l}$a$\end{tabular}}}}%
    \put(0,0){\includegraphics[width=\unitlength,page=2]{rectangles.pdf}}%
    \put(0.58483276,0.27541366){\color[rgb]{0,0,0}\makebox(0,0)[lt]{\lineheight{0}\smash{\begin{tabular}[t]{l}$R_1$\end{tabular}}}}%
    \put(0.63068759,0.27541366){\color[rgb]{0,0,0}\makebox(0,0)[lt]{\lineheight{0}\smash{\begin{tabular}[t]{l}$R_2$\end{tabular}}}}%
    \put(0.67140778,0.27541366){\color[rgb]{0,0,0}\makebox(0,0)[lt]{\lineheight{0}\smash{\begin{tabular}[t]{l}$R_3$\end{tabular}}}}%
    \put(0.71590319,0.27541366){\color[rgb]{0,0,0}\makebox(0,0)[lt]{\lineheight{0}\smash{\begin{tabular}[t]{l}$R_4$\end{tabular}}}}%
    \put(0.78529342,0.27541366){\color[rgb]{0,0,0}\makebox(0,0)[lt]{\lineheight{0}\smash{\begin{tabular}[t]{l}$R_5$\end{tabular}}}}%
    \put(0.93631855,0.27541366){\color[rgb]{0,0,0}\makebox(0,0)[lt]{\lineheight{0}\smash{\begin{tabular}[t]{l}$R_6$\end{tabular}}}}%
    \put(0.66150957,0.24339349){\color[rgb]{0,0,0}\makebox(0,0)[lt]{\lineheight{0}\smash{\begin{tabular}[t]{l}$R_7$\end{tabular}}}}%
    \put(0.70226579,0.24339352){\color[rgb]{0,0,0}\makebox(0,0)[lt]{\lineheight{0}\smash{\begin{tabular}[t]{l}$R_8$\end{tabular}}}}%
    \put(0.54283209,0.38014436){\color[rgb]{0,0,0}\makebox(0,0)[lt]{\lineheight{0}\smash{\begin{tabular}[t]{l}$a$\end{tabular}}}}%
    \put(0.97940355,0.44306358){\color[rgb]{0,0,0}\makebox(0,0)[lt]{\lineheight{0}\smash{\begin{tabular}[t]{l}$a$\end{tabular}}}}%
    \put(0.97940353,0.27480588){\color[rgb]{0,0,0}\makebox(0,0)[lt]{\lineheight{0}\smash{\begin{tabular}[t]{l}$b$\end{tabular}}}}%
    \put(0.6690486,0.36529013){\color[rgb]{0,0,0}\makebox(0,0)[lt]{\lineheight{0}\smash{\begin{tabular}[t]{l}$c$\end{tabular}}}}%
    \put(0.69953204,0.10265889){\color[rgb]{0,0,0}\makebox(0,0)[lt]{\lineheight{0}\smash{\begin{tabular}[t]{l}$d$\end{tabular}}}}%
    \put(0.52978674,0.28164032){\color[rgb]{0,0,0}\makebox(0,0)[lt]{\lineheight{0}\smash{\begin{tabular}[t]{l}$\tau_+$\end{tabular}}}}%
    \put(0.52951217,0.23716682){\color[rgb]{0,0,0}\makebox(0,0)[lt]{\lineheight{0}\smash{\begin{tabular}[t]{l}$\tau_-$\end{tabular}}}}%
    \put(0.75487276,0.38383272){\color[rgb]{0,0,0}\makebox(0,0)[lt]{\lineheight{0}\smash{\begin{tabular}[t]{l}$e$\end{tabular}}}}%
    \put(0,0){\includegraphics[width=\unitlength,page=3]{rectangles.pdf}}%
    \put(0.75523671,0.55831338){\color[rgb]{0,0,0}\makebox(0,0)[lt]{\lineheight{0}\smash{\begin{tabular}[t]{l}$b$\end{tabular}}}}%
  \end{picture}%
\endgroup%

%% file: figure/decomposition.pdf_tex
%% Creator: Inkscape 1.3 (0e150ed6c4, 2023-07-21), www.inkscape.org
%% Accompanies image file 'decomposition.pdf' (pdf, eps, ps)
\begingroup%
  \makeatletter%
  \providecommand\color[2][]{%
    \errmessage{(Inkscape) Color is used for the text in Inkscape, but the package 'color.sty' is not loaded}%
    \renewcommand\color[2][]{}%
  }%
  \providecommand\transparent[1]{%
    \errmessage{(Inkscape) Transparency is used (non-zero) for the text in Inkscape, but the package 'transparent.sty' is not loaded}%
    \renewcommand\transparent[1]{}%
  }%
  \providecommand\rotatebox[2]{#2}%
  \newcommand*\fsize{\dimexpr\f@size pt\relax}%
  \newcommand*\lineheight[1]{\fontsize{\fsize}{#1\fsize}\selectfont}%
  \ifx\svgwidth\undefined%
    \setlength{\unitlength}{541.4514064bp}%
    \ifx\svgscale\undefined%
      \relax%
    \else%
      \setlength{\unitlength}{\unitlength * \real{\svgscale}}%
    \fi%
  \else%
    \setlength{\unitlength}{\svgwidth}%
  \fi%
  \global\let\svgwidth\undefined%
  \global\let\svgscale\undefined%
  \makeatother%
  \begin{picture}(1,0.50804627)%
    \lineheight{1}%
    \setlength\tabcolsep{0pt}%
    \put(0,0){\includegraphics[width=\unitlength,page=1]{decomposition.pdf}}%
    \put(0.53387618,0.26859195){\color[rgb]{0,0,0}\makebox(0,0)[lt]{\lineheight{0}\smash{\begin{tabular}[t]{l}$R_1$\end{tabular}}}}%
    \put(0.55987018,0.20794777){\color[rgb]{0,0,0}\makebox(0,0)[lt]{\lineheight{0}\smash{\begin{tabular}[t]{l}$R_2$\end{tabular}}}}%
    \put(0,0){\includegraphics[width=\unitlength,page=2]{decomposition.pdf}}%
    \put(0.58586418,0.26859195){\color[rgb]{0,0,0}\makebox(0,0)[lt]{\lineheight{0}\smash{\begin{tabular}[t]{l}$R_6$\end{tabular}}}}%
    \put(0,0){\includegraphics[width=\unitlength,page=3]{decomposition.pdf}}%
    \put(0.56132214,0.32923614){\color[rgb]{0,0,0}\makebox(0,0)[lt]{\lineheight{0}\smash{\begin{tabular}[t]{l}$R_8$\end{tabular}}}}%
    \put(0.49164735,0.28996846){\color[rgb]{0,0,0}\makebox(0,0)[lt]{\lineheight{0}\smash{\begin{tabular}[t]{l}$a$\end{tabular}}}}%
    \put(0,0){\includegraphics[width=\unitlength,page=4]{decomposition.pdf}}%
    \put(0.64225825,0.28996846){\color[rgb]{0,0,0}\makebox(0,0)[lt]{\lineheight{0}\smash{\begin{tabular}[t]{l}$a$\end{tabular}}}}%
    \put(0.66317723,0.22022123){\color[rgb]{0,0,0}\makebox(0,0)[lt]{\lineheight{0}\smash{\begin{tabular}[t]{l}$b$\end{tabular}}}}%
    \put(0,0){\includegraphics[width=\unitlength,page=5]{decomposition.pdf}}%
    \put(0.04762987,0.23733126){\color[rgb]{0,0,0}\makebox(0,0)[lt]{\lineheight{0}\smash{\begin{tabular}[t]{l}$R_1$\end{tabular}}}}%
    \put(0.08763591,0.23733126){\color[rgb]{0,0,0}\makebox(0,0)[lt]{\lineheight{0}\smash{\begin{tabular}[t]{l}$R_2$\end{tabular}}}}%
    \put(0.1231622,0.23733126){\color[rgb]{0,0,0}\makebox(0,0)[lt]{\lineheight{0}\smash{\begin{tabular}[t]{l}$R_3$\end{tabular}}}}%
    \put(0.16198221,0.23733126){\color[rgb]{0,0,0}\makebox(0,0)[lt]{\lineheight{0}\smash{\begin{tabular}[t]{l}$R_4$\end{tabular}}}}%
    \put(0.22252167,0.23733126){\color[rgb]{0,0,0}\makebox(0,0)[lt]{\lineheight{0}\smash{\begin{tabular}[t]{l}$R_5$\end{tabular}}}}%
    \put(0.27875117,0.23733126){\color[rgb]{0,0,0}\makebox(0,0)[lt]{\lineheight{0}\smash{\begin{tabular}[t]{l}$R_3$\end{tabular}}}}%
    \put(0.31150702,0.23733126){\color[rgb]{0,0,0}\makebox(0,0)[lt]{\lineheight{0}\smash{\begin{tabular}[t]{l}$R_2$\end{tabular}}}}%
    \put(0.3542835,0.23733126){\color[rgb]{0,0,0}\makebox(0,0)[lt]{\lineheight{0}\smash{\begin{tabular}[t]{l}$R_6$\end{tabular}}}}%
    \put(0.0586184,0.20939527){\color[rgb]{0,0,0}\makebox(0,0)[lt]{\lineheight{0}\smash{\begin{tabular}[t]{l}$R_5$\end{tabular}}}}%
    \put(0.11452655,0.20939527){\color[rgb]{0,0,0}\makebox(0,0)[lt]{\lineheight{0}\smash{\begin{tabular}[t]{l}$R_7$\end{tabular}}}}%
    \put(0.15008429,0.20939527){\color[rgb]{0,0,0}\makebox(0,0)[lt]{\lineheight{0}\smash{\begin{tabular}[t]{l}$R_8$\end{tabular}}}}%
    \put(0.19018655,0.20939527){\color[rgb]{0,0,0}\makebox(0,0)[lt]{\lineheight{0}\smash{\begin{tabular}[t]{l}$R_6$\end{tabular}}}}%
    \put(0.23863318,0.20939527){\color[rgb]{0,0,0}\makebox(0,0)[lt]{\lineheight{0}\smash{\begin{tabular}[t]{l}$R_1$\end{tabular}}}}%
    \put(0.28365407,0.20939528){\color[rgb]{0,0,0}\makebox(0,0)[lt]{\lineheight{0}\smash{\begin{tabular}[t]{l}$R_8$\end{tabular}}}}%
    \put(0.31898438,0.20939528){\color[rgb]{0,0,0}\makebox(0,0)[lt]{\lineheight{0}\smash{\begin{tabular}[t]{l}$R_7$\end{tabular}}}}%
    \put(0.353311,0.20939528){\color[rgb]{0,0,0}\makebox(0,0)[lt]{\lineheight{0}\smash{\begin{tabular}[t]{l}$R_4$\end{tabular}}}}%
    \put(0.01375675,0.32593317){\color[rgb]{0,0,0}\makebox(0,0)[lt]{\lineheight{0}\smash{\begin{tabular}[t]{l}$a$\end{tabular}}}}%
    \put(0.01375675,0.21095386){\color[rgb]{0,0,0}\makebox(0,0)[lt]{\lineheight{0}\smash{\begin{tabular}[t]{l}$b$\end{tabular}}}}%
    \put(0.39187299,0.38359735){\color[rgb]{0,0,0}\makebox(0,0)[lt]{\lineheight{0}\smash{\begin{tabular}[t]{l}$a$\end{tabular}}}}%
    \put(0.39187299,0.236801){\color[rgb]{0,0,0}\makebox(0,0)[lt]{\lineheight{0}\smash{\begin{tabular}[t]{l}$b$\end{tabular}}}}%
    \put(0.11002259,0.31851427){\color[rgb]{0,0,0}\makebox(0,0)[lt]{\lineheight{0}\smash{\begin{tabular}[t]{l}$c$\end{tabular}}}}%
    \put(0.29675363,0.31851427){\color[rgb]{0,0,0}\makebox(0,0)[lt]{\lineheight{0}\smash{\begin{tabular}[t]{l}$c$\end{tabular}}}}%
    \put(0.13661783,0.07830044){\color[rgb]{0,0,0}\makebox(0,0)[lt]{\lineheight{0}\smash{\begin{tabular}[t]{l}$d$\end{tabular}}}}%
    \put(0.30684354,0.10918368){\color[rgb]{0,0,0}\makebox(0,0)[lt]{\lineheight{0}\smash{\begin{tabular}[t]{l}$d$\end{tabular}}}}%
    \put(-0.00039499,0.24276371){\color[rgb]{0,0,0}\makebox(0,0)[lt]{\lineheight{0}\smash{\begin{tabular}[t]{l}$\tau_+$\end{tabular}}}}%
    \put(-0.00039499,0.193537){\color[rgb]{0,0,0}\makebox(0,0)[lt]{\lineheight{0}\smash{\begin{tabular}[t]{l}$\tau_-$\end{tabular}}}}%
    \put(0.19598118,0.33192141){\color[rgb]{0,0,0}\makebox(0,0)[lt]{\lineheight{0}\smash{\begin{tabular}[t]{l}$e$\end{tabular}}}}%
    \put(0,0){\includegraphics[width=\unitlength,page=6]{decomposition.pdf}}%
    \put(0.74997121,0.34833269){\color[rgb]{0,0,0}\makebox(0,0)[lt]{\lineheight{0}\smash{\begin{tabular}[t]{l}$b$\end{tabular}}}}%
    \put(0.84565379,0.26584628){\color[rgb]{0,0,0}\makebox(0,0)[lt]{\lineheight{0}\smash{\begin{tabular}[t]{l}$R_5$\end{tabular}}}}%
    \put(0,0){\includegraphics[width=\unitlength,page=7]{decomposition.pdf}}%
    \put(0.89401917,0.1976449){\color[rgb]{0,0,0}\makebox(0,0)[lt]{\lineheight{0}\smash{\begin{tabular}[t]{l}$R_3$\end{tabular}}}}%
    \put(0,0){\includegraphics[width=\unitlength,page=8]{decomposition.pdf}}%
    \put(0.91064123,0.26967351){\color[rgb]{0,0,0}\makebox(0,0)[lt]{\lineheight{0}\smash{\begin{tabular}[t]{l}$R_4$\end{tabular}}}}%
    \put(0,0){\includegraphics[width=\unitlength,page=9]{decomposition.pdf}}%
    \put(0.8912488,0.34170205){\color[rgb]{0,0,0}\makebox(0,0)[lt]{\lineheight{0}\smash{\begin{tabular}[t]{l}$R_7$\end{tabular}}}}%
    \put(0.9508104,0.26551795){\color[rgb]{0,0,0}\makebox(0,0)[lt]{\lineheight{0}\smash{\begin{tabular}[t]{l}$e$\end{tabular}}}}%
    \put(0.79705772,0.27382894){\color[rgb]{0,0,0}\makebox(0,0)[lt]{\lineheight{0}\smash{\begin{tabular}[t]{l}$e$\end{tabular}}}}%
    \put(0.5672941,0.11080681){\color[rgb]{0,0,0}\makebox(0,0)[lt]{\lineheight{0}\smash{\begin{tabular}[t]{l}$X_1$\end{tabular}}}}%
    \put(0.88877633,0.1108068){\color[rgb]{0,0,0}\makebox(0,0)[lt]{\lineheight{0}\smash{\begin{tabular}[t]{l}$X_2$\end{tabular}}}}%
  \end{picture}%
\endgroup%

%% file: figure/example.pdf_tex
%% Creator: Inkscape 1.3 (0e150ed6c4, 2023-07-21), www.inkscape.org
%% Accompanies image file 'example.pdf' (pdf, eps, ps)
\begingroup%
  \makeatletter%
  \providecommand\color[2][]{%
    \errmessage{(Inkscape) Color is used for the text in Inkscape, but the package 'color.sty' is not loaded}%
    \renewcommand\color[2][]{}%
  }%
  \providecommand\transparent[1]{%
    \errmessage{(Inkscape) Transparency is used (non-zero) for the text in Inkscape, but the package 'transparent.sty' is not loaded}%
    \renewcommand\transparent[1]{}%
  }%
  \providecommand\rotatebox[2]{#2}%
  \newcommand*\fsize{\dimexpr\f@size pt\relax}%
  \newcommand*\lineheight[1]{\fontsize{\fsize}{#1\fsize}\selectfont}%
  \ifx\svgwidth\undefined%
    \setlength{\unitlength}{663.93750288bp}%
    \ifx\svgscale\undefined%
      \relax%
    \else%
      \setlength{\unitlength}{\unitlength * \real{\svgscale}}%
    \fi%
  \else%
    \setlength{\unitlength}{\svgwidth}%
  \fi%
  \global\let\svgwidth\undefined%
  \global\let\svgscale\undefined%
  \makeatother%
  \begin{picture}(1,0.29226644)%
    \lineheight{1}%
    \setlength\tabcolsep{0pt}%
    \put(0,0){\includegraphics[width=\unitlength,page=1]{example.pdf}}%
    \put(0.02801668,0.12644805){\color[rgb]{0,0,0}\makebox(0,0)[lt]{\lineheight{0}\smash{\begin{tabular}[t]{l}$a$\end{tabular}}}}%
    \put(0.07194312,0.20505455){\color[rgb]{0,0,0}\makebox(0,0)[lt]{\lineheight{0}\smash{\begin{tabular}[t]{l}$b$\end{tabular}}}}%
    \put(0.07144228,0.12644805){\color[rgb]{0,0,0}\makebox(0,0)[lt]{\lineheight{0}\smash{\begin{tabular}[t]{l}$A$\end{tabular}}}}%
    \put(0.11912758,0.12644805){\color[rgb]{0,0,0}\makebox(0,0)[lt]{\lineheight{0}\smash{\begin{tabular}[t]{l}$B$\end{tabular}}}}%
    \put(0.10993169,0.14739011){\color[rgb]{0,0,0}\makebox(0,0)[lt]{\lineheight{0}\smash{\begin{tabular}[t]{l}$c$\end{tabular}}}}%
    \put(0.19398234,0.12702168){\color[rgb]{0,0,0}\makebox(0,0)[lt]{\lineheight{0}\smash{\begin{tabular}[t]{l}$C$\end{tabular}}}}%
    \put(0.1575176,0.17238351){\color[rgb]{0,0,0}\makebox(0,0)[lt]{\lineheight{0}\smash{\begin{tabular}[t]{l}$d$\end{tabular}}}}%
    \put(0.1575176,0.10218504){\color[rgb]{0,0,0}\makebox(0,0)[lt]{\lineheight{0}\smash{\begin{tabular}[t]{l}$e$\end{tabular}}}}%
    \put(0.23604661,0.12702168){\color[rgb]{0,0,0}\makebox(0,0)[lt]{\lineheight{0}\smash{\begin{tabular}[t]{l}$f$\end{tabular}}}}%
    \put(0,0){\includegraphics[width=\unitlength,page=2]{example.pdf}}%
    \put(0.3129184,0.2356664){\color[rgb]{0,0,0}\makebox(0,0)[lt]{\lineheight{0}\smash{\begin{tabular}[t]{l}$a$\end{tabular}}}}%
    \put(0.3355984,0.2356664){\color[rgb]{0,0,0}\makebox(0,0)[lt]{\lineheight{0}\smash{\begin{tabular}[t]{l}$A$\end{tabular}}}}%
    \put(0.36605781,0.2356664){\color[rgb]{0,0,0}\makebox(0,0)[lt]{\lineheight{0}\smash{\begin{tabular}[t]{l}$b$\end{tabular}}}}%
    \put(0.39777481,0.2356664){\color[rgb]{0,0,0}\makebox(0,0)[lt]{\lineheight{0}\smash{\begin{tabular}[t]{l}$A$\end{tabular}}}}%
    \put(0.41919722,0.2356664){\color[rgb]{0,0,0}\makebox(0,0)[lt]{\lineheight{0}\smash{\begin{tabular}[t]{l}$c$\end{tabular}}}}%
    \put(0.43911052,0.2356664){\color[rgb]{0,0,0}\makebox(0,0)[lt]{\lineheight{0}\smash{\begin{tabular}[t]{l}$B$\end{tabular}}}}%
    \put(0.47182918,0.2356664){\color[rgb]{0,0,0}\makebox(0,0)[lt]{\lineheight{0}\smash{\begin{tabular}[t]{l}$d$\end{tabular}}}}%
    \put(0.50354618,0.2356664){\color[rgb]{0,0,0}\makebox(0,0)[lt]{\lineheight{0}\smash{\begin{tabular}[t]{l}$C$\end{tabular}}}}%
    \put(0.54128196,0.2356664){\color[rgb]{0,0,0}\makebox(0,0)[lt]{\lineheight{0}\smash{\begin{tabular}[t]{l}$f$\end{tabular}}}}%
    \put(0,0){\includegraphics[width=\unitlength,page=3]{example.pdf}}%
    \put(0.50691152,0.17672742){\color[rgb]{0,0,0}\makebox(0,0)[lt]{\lineheight{0}\smash{\begin{tabular}[t]{l}$A$\end{tabular}}}}%
    \put(0.54191002,0.17672742){\color[rgb]{0,0,0}\makebox(0,0)[lt]{\lineheight{0}\smash{\begin{tabular}[t]{l}$a$\end{tabular}}}}%
    \put(0.47568138,0.17672742){\color[rgb]{0,0,0}\makebox(0,0)[lt]{\lineheight{0}\smash{\begin{tabular}[t]{l}$c$\end{tabular}}}}%
    \put(0.44068287,0.17672742){\color[rgb]{0,0,0}\makebox(0,0)[lt]{\lineheight{0}\smash{\begin{tabular}[t]{l}$B$\end{tabular}}}}%
    \put(0.40668603,0.17672742){\color[rgb]{0,0,0}\makebox(0,0)[lt]{\lineheight{0}\smash{\begin{tabular}[t]{l}$e$\end{tabular}}}}%
    \put(0.36892965,0.17672742){\color[rgb]{0,0,0}\makebox(0,0)[lt]{\lineheight{0}\smash{\begin{tabular}[t]{l}$C$\end{tabular}}}}%
    \put(0.33292949,0.17672742){\color[rgb]{0,0,0}\makebox(0,0)[lt]{\lineheight{0}\smash{\begin{tabular}[t]{l}$f$\end{tabular}}}}%
    \put(0,0){\includegraphics[width=\unitlength,page=4]{example.pdf}}%
    \put(0.36066328,0.13779957){\color[rgb]{0,0,0}\makebox(0,0)[lt]{\lineheight{0}\smash{\begin{tabular}[t]{l}$b$\end{tabular}}}}%
    \put(0.31241757,0.13766525){\color[rgb]{0,0,0}\makebox(0,0)[lt]{\lineheight{0}\smash{\begin{tabular}[t]{l}$A$\end{tabular}}}}%
    \put(0,0){\includegraphics[width=\unitlength,page=5]{example.pdf}}%
    \put(0.32776236,0.08870494){\color[rgb]{0,0,0}\makebox(0,0)[lt]{\lineheight{0}\smash{\begin{tabular}[t]{l}$B$\end{tabular}}}}%
    \put(0.32751305,0.03559323){\color[rgb]{0,0,0}\makebox(0,0)[lt]{\lineheight{0}\smash{\begin{tabular}[t]{l}$C$\end{tabular}}}}%
    \put(0.31643326,0.06214909){\color[rgb]{0,0,0}\makebox(0,0)[lt]{\lineheight{0}\smash{\begin{tabular}[t]{l}$e$\end{tabular}}}}%
    \put(0.36066328,0.06264207){\color[rgb]{0,0,0}\makebox(0,0)[lt]{\lineheight{0}\smash{\begin{tabular}[t]{l}$d$\end{tabular}}}}%
  \end{picture}%
\endgroup%

%% file: figure/metricgraph.pdf_tex
%% Creator: Inkscape 1.3 (0e150ed6c4, 2023-07-21), www.inkscape.org
%% Accompanies image file 'metricgraph.pdf' (pdf, eps, ps)
\begingroup%
  \makeatletter%
  \providecommand\color[2][]{%
    \errmessage{(Inkscape) Color is used for the text in Inkscape, but the package 'color.sty' is not loaded}%
    \renewcommand\color[2][]{}%
  }%
  \providecommand\transparent[1]{%
    \errmessage{(Inkscape) Transparency is used (non-zero) for the text in Inkscape, but the package 'transparent.sty' is not loaded}%
    \renewcommand\transparent[1]{}%
  }%
  \providecommand\rotatebox[2]{#2}%
  \newcommand*\fsize{\dimexpr\f@size pt\relax}%
  \newcommand*\lineheight[1]{\fontsize{\fsize}{#1\fsize}\selectfont}%
  \ifx\svgwidth\undefined%
    \setlength{\unitlength}{538.41987189bp}%
    \ifx\svgscale\undefined%
      \relax%
    \else%
      \setlength{\unitlength}{\unitlength * \real{\svgscale}}%
    \fi%
  \else%
    \setlength{\unitlength}{\svgwidth}%
  \fi%
  \global\let\svgwidth\undefined%
  \global\let\svgscale\undefined%
  \makeatother%
  \begin{picture}(1,0.79497402)%
    \lineheight{1}%
    \setlength\tabcolsep{0pt}%
    \put(0,0){\includegraphics[width=\unitlength,page=1]{metricgraph.pdf}}%
    \put(0.56863233,0.61177031){\color[rgb]{0,0,0}\makebox(0,0)[lt]{\lineheight{0}\smash{\begin{tabular}[t]{l}$A$\end{tabular}}}}%
    \put(0.56863233,0.66720801){\color[rgb]{0,0,0}\makebox(0,0)[lt]{\lineheight{0}\smash{\begin{tabular}[t]{l}$B$\end{tabular}}}}%
    \put(0.56902954,0.64185373){\color[rgb]{0,0,0}\makebox(0,0)[lt]{\lineheight{0}\smash{\begin{tabular}[t]{l}$b$\end{tabular}}}}%
    \put(0.63875439,0.64048338){\color[rgb]{0,0,0}\makebox(0,0)[lt]{\lineheight{0}\smash{\begin{tabular}[t]{l}$d$\end{tabular}}}}%
    \put(0,0){\includegraphics[width=\unitlength,page=2]{metricgraph.pdf}}%
    \put(0.56743912,0.4881077){\color[rgb]{0,0,0}\makebox(0,0)[lt]{\lineheight{0}\smash{\begin{tabular}[t]{l}$B$\end{tabular}}}}%
    \put(0.56743912,0.53764161){\color[rgb]{0,0,0}\makebox(0,0)[lt]{\lineheight{0}\smash{\begin{tabular}[t]{l}$A$\end{tabular}}}}%
    \put(0.56805671,0.51328224){\color[rgb]{0,0,0}\makebox(0,0)[lt]{\lineheight{0}\smash{\begin{tabular}[t]{l}$b$\end{tabular}}}}%
    \put(0.63822361,0.51174935){\color[rgb]{0,0,0}\makebox(0,0)[lt]{\lineheight{0}\smash{\begin{tabular}[t]{l}$c$\end{tabular}}}}%
    \put(0,0){\includegraphics[width=\unitlength,page=3]{metricgraph.pdf}}%
    \put(0.62800903,0.71625057){\color[rgb]{0,0,0}\makebox(0,0)[lt]{\lineheight{0}\smash{\begin{tabular}[t]{l}$A$\end{tabular}}}}%
    \put(0.74708658,0.71625057){\color[rgb]{0,0,0}\makebox(0,0)[lt]{\lineheight{0}\smash{\begin{tabular}[t]{l}$B$\end{tabular}}}}%
    \put(0.8768833,0.71625057){\color[rgb]{0,0,0}\makebox(0,0)[lt]{\lineheight{0}\smash{\begin{tabular}[t]{l}$A$\end{tabular}}}}%
    \put(0.56773839,0.71774863){\color[rgb]{0,0,0}\makebox(0,0)[lt]{\lineheight{0}\smash{\begin{tabular}[t]{l}$a$\end{tabular}}}}%
    \put(0.68933982,0.71774863){\color[rgb]{0,0,0}\makebox(0,0)[lt]{\lineheight{0}\smash{\begin{tabular}[t]{l}$d$\end{tabular}}}}%
    \put(0.80998974,0.71774863){\color[rgb]{0,0,0}\makebox(0,0)[lt]{\lineheight{0}\smash{\begin{tabular}[t]{l}$c$\end{tabular}}}}%
    \put(0.92979059,0.71774863){\color[rgb]{0,0,0}\makebox(0,0)[lt]{\lineheight{0}\smash{\begin{tabular}[t]{l}$a$\end{tabular}}}}%
    \put(0,0){\includegraphics[width=\unitlength,page=4]{metricgraph.pdf}}%
    \put(0.5655859,0.14709191){\color[rgb]{0,0,0}\makebox(0,0)[lt]{\lineheight{0}\smash{\begin{tabular}[t]{l}$B$\end{tabular}}}}%
    \put(0.56279997,0.19914456){\color[rgb]{0,0,0}\makebox(0,0)[lt]{\lineheight{0}\smash{\begin{tabular}[t]{l}$A$\end{tabular}}}}%
    \put(0.56403514,0.17202882){\color[rgb]{0,0,0}\makebox(0,0)[lt]{\lineheight{0}\smash{\begin{tabular}[t]{l}$b$\end{tabular}}}}%
    \put(0.63482937,0.18009992){\color[rgb]{0,0,0}\makebox(0,0)[lt]{\lineheight{0}\smash{\begin{tabular}[t]{l}$c$\end{tabular}}}}%
    \put(0,0){\includegraphics[width=\unitlength,page=5]{metricgraph.pdf}}%
    \put(0.63379812,0.04700818){\color[rgb]{0,0,0}\makebox(0,0)[lt]{\lineheight{0}\smash{\begin{tabular}[t]{l}$c$\end{tabular}}}}%
    \put(0.56733232,0.09368047){\color[rgb]{0,0,0}\makebox(0,0)[lt]{\lineheight{0}\smash{\begin{tabular}[t]{l}$B$\end{tabular}}}}%
    \put(0.56843488,0.0098105){\color[rgb]{0,0,0}\makebox(0,0)[lt]{\lineheight{0}\smash{\begin{tabular}[t]{l}$A$\end{tabular}}}}%
    \put(0.56300177,0.05174549){\color[rgb]{0,0,0}\makebox(0,0)[lt]{\lineheight{0}\smash{\begin{tabular}[t]{l}$d$\end{tabular}}}}%
    \put(0,0){\includegraphics[width=\unitlength,page=6]{metricgraph.pdf}}%
    \put(0.62834185,0.25583106){\color[rgb]{0,0,0}\makebox(0,0)[lt]{\lineheight{0}\smash{\begin{tabular}[t]{l}$A$\end{tabular}}}}%
    \put(0.74777287,0.25583106){\color[rgb]{0,0,0}\makebox(0,0)[lt]{\lineheight{0}\smash{\begin{tabular}[t]{l}$B$\end{tabular}}}}%
    \put(0.87850422,0.25583106){\color[rgb]{0,0,0}\makebox(0,0)[lt]{\lineheight{0}\smash{\begin{tabular}[t]{l}$A$\end{tabular}}}}%
    \put(0.56445531,0.25797492){\color[rgb]{0,0,0}\makebox(0,0)[lt]{\lineheight{0}\smash{\begin{tabular}[t]{l}$a$\end{tabular}}}}%
    \put(0.68849301,0.25797492){\color[rgb]{0,0,0}\makebox(0,0)[lt]{\lineheight{0}\smash{\begin{tabular}[t]{l}$b$\end{tabular}}}}%
    \put(0.81393599,0.25797492){\color[rgb]{0,0,0}\makebox(0,0)[lt]{\lineheight{0}\smash{\begin{tabular}[t]{l}$d$\end{tabular}}}}%
    \put(0.93495025,0.25797492){\color[rgb]{0,0,0}\makebox(0,0)[lt]{\lineheight{0}\smash{\begin{tabular}[t]{l}$a$\end{tabular}}}}%
    \put(0,0){\includegraphics[width=\unitlength,page=7]{metricgraph.pdf}}%
    \put(0.21098702,0.65227151){\color[rgb]{0,0,0}\makebox(0,0)[lt]{\lineheight{0}\smash{\begin{tabular}[t]{l}$A$\end{tabular}}}}%
    \put(0.2059369,0.68686783){\color[rgb]{0,0,0}\makebox(0,0)[lt]{\lineheight{0}\smash{\begin{tabular}[t]{l}$B$\end{tabular}}}}%
    \put(0.11786149,0.68965376){\color[rgb]{0,0,0}\makebox(0,0)[lt]{\lineheight{0}\smash{\begin{tabular}[t]{l}$c$\end{tabular}}}}%
    \put(0.31790088,0.69258252){\color[rgb]{0,0,0}\makebox(0,0)[lt]{\lineheight{0}\smash{\begin{tabular}[t]{l}$d$\end{tabular}}}}%
    \put(0.24168504,0.66791045){\color[rgb]{0,0,0}\makebox(0,0)[lt]{\lineheight{0}\smash{\begin{tabular}[t]{l}$b$\end{tabular}}}}%
    \put(0,0){\includegraphics[width=\unitlength,page=8]{metricgraph.pdf}}%
    \put(0.21160461,0.57041309){\color[rgb]{0,0,0}\makebox(0,0)[lt]{\lineheight{0}\smash{\begin{tabular}[t]{l}$a$\end{tabular}}}}%
    \put(0,0){\includegraphics[width=\unitlength,page=9]{metricgraph.pdf}}%
    \put(0.13600849,0.17627029){\color[rgb]{0,0,0}\makebox(0,0)[lt]{\lineheight{0}\smash{\begin{tabular}[t]{l}$A$\end{tabular}}}}%
    \put(0.14558618,0.20956729){\color[rgb]{0,0,0}\makebox(0,0)[lt]{\lineheight{0}\smash{\begin{tabular}[t]{l}$B$\end{tabular}}}}%
    \put(0.0867668,0.20956729){\color[rgb]{0,0,0}\makebox(0,0)[lt]{\lineheight{0}\smash{\begin{tabular}[t]{l}$c$\end{tabular}}}}%
    \put(0.26578453,0.19718662){\color[rgb]{0,0,0}\makebox(0,0)[lt]{\lineheight{0}\smash{\begin{tabular}[t]{l}$d$\end{tabular}}}}%
    \put(0,0){\includegraphics[width=\unitlength,page=10]{metricgraph.pdf}}%
    \put(0.17439665,0.1889158){\color[rgb]{0,0,0}\makebox(0,0)[lt]{\lineheight{0}\smash{\begin{tabular}[t]{l}$b$\end{tabular}}}}%
    \put(0.20187916,0.1240557){\color[rgb]{0,0,0}\makebox(0,0)[lt]{\lineheight{0}\smash{\begin{tabular}[t]{l}$a$\end{tabular}}}}%
  \end{picture}%
\endgroup%

%% file: figure/pickpoints.pdf_tex
%% Creator: Inkscape 1.3 (0e150ed6c4, 2023-07-21), www.inkscape.org
%% Accompanies image file 'pickpoints.pdf' (pdf, eps, ps)
\begingroup%
  \makeatletter%
  \providecommand\color[2][]{%
    \errmessage{(Inkscape) Color is used for the text in Inkscape, but the package 'color.sty' is not loaded}%
    \renewcommand\color[2][]{}%
  }%
  \providecommand\transparent[1]{%
    \errmessage{(Inkscape) Transparency is used (non-zero) for the text in Inkscape, but the package 'transparent.sty' is not loaded}%
    \renewcommand\transparent[1]{}%
  }%
  \providecommand\rotatebox[2]{#2}%
  \newcommand*\fsize{\dimexpr\f@size pt\relax}%
  \newcommand*\lineheight[1]{\fontsize{\fsize}{#1\fsize}\selectfont}%
  \ifx\svgwidth\undefined%
    \setlength{\unitlength}{638.04065134bp}%
    \ifx\svgscale\undefined%
      \relax%
    \else%
      \setlength{\unitlength}{\unitlength * \real{\svgscale}}%
    \fi%
  \else%
    \setlength{\unitlength}{\svgwidth}%
  \fi%
  \global\let\svgwidth\undefined%
  \global\let\svgscale\undefined%
  \makeatother%
  \begin{picture}(1,0.32683376)%
    \lineheight{1}%
    \setlength\tabcolsep{0pt}%
    \put(0,0){\includegraphics[width=\unitlength,page=1]{pickpoints.pdf}}%
    \put(0.08407925,0.15143135){\color[rgb]{0,0,0}\makebox(0,0)[lt]{\lineheight{0}\smash{\begin{tabular}[t]{l}$R_i$\end{tabular}}}}%
    \put(0,0){\includegraphics[width=\unitlength,page=2]{pickpoints.pdf}}%
    \put(0.36847771,0.1472253){\color[rgb]{0,0,0}\makebox(0,0)[lt]{\lineheight{0}\smash{\begin{tabular}[t]{l}$R_i$\end{tabular}}}}%
    \put(0,0){\includegraphics[width=\unitlength,page=3]{pickpoints.pdf}}%
    \put(0.62590528,0.16192797){\color[rgb]{0,0,0}\makebox(0,0)[lt]{\lineheight{0}\smash{\begin{tabular}[t]{l}$R_i^\prime$\end{tabular}}}}%
    \put(0,0){\includegraphics[width=\unitlength,page=4]{pickpoints.pdf}}%
    \put(0.94533132,0.14500685){\color[rgb]{0,0,0}\makebox(0,0)[lt]{\lineheight{0}\smash{\begin{tabular}[t]{l}$R_i^\prime$\end{tabular}}}}%
  \end{picture}%
\endgroup%